\documentclass[12pt,reqno]{amsart}
\usepackage{amsmath,amsfonts,amsthm,amssymb}
\usepackage[usenames,dvipsnames]{xcolor}
\usepackage[T1]{fontenc}
\usepackage[shortlabels] {enumitem}
\usepackage[plain]{fancyref}

\usepackage[
abbrev,
]{amsrefs}

\usepackage{bm}
\usepackage[
bookmarks=true,         %generates bookmarks for all entries in the "Table of Contents"
bookmarksnumbered=true, %includes chapter-/header-/section-/subsection-/... -
colorlinks=true,
pdfstartview=FitV,
linkcolor=blue, citecolor=blue, urlcolor=blue]{hyperref}
\usepackage{fullpage}
\newtheorem{theorem}{Theorem}[section]

\newtheorem{definition}[theorem]{Definition}

\newtheorem{lemma}[theorem]{Lemma}

\newtheorem{proposition}[theorem]{Proposition}
\theoremstyle{remark}
\newtheorem{remark}[theorem]{Remark}

\newcommand*{\fancyrefproplabelprefix}{prop}

\frefformat{prop}{\fancyrefproplabelprefix}{Proposition~#1#2%
}%
\newcommand{\ep}{\varepsilon}

\newcommand{\laa}{\Lambda}

\newcommand{\oom}{\Omega}

\newcommand{\si}{\sigma}

\newcommand{\cf}{\mathcal F}

\newcommand{\HH}{\mathfrak H}
\newcommand{\lp}{\left(}
\newcommand{\rp}{\right)}
\newcommand{\lc}{\left[}
\newcommand{\rc}{\right]}

\def\RR{\mathbb{R}}
\def\NN{\mathbb{N}}
\def\EE{\mathbb{E}}

\def\cee{\mathcal{E}}
\def\css{\mathcal{S}}
\def\cff{\mathcal{F}}

\def\lt{\left}
\def\rt{\right}
\begin{document}
\author[J. Huang, K. L\^e, D. Nualart]{Jingyu Huang \and Khoa L\^e \and David Nualart}
\address{Jingyu Huang and Khoa L\^e: Mathematical Sciences Research Institute, Berkeley, California, 94720}
\email{jhuang@msri.org, khoa.le@ucalgary.ca}
% \author[K. L\^e]{Khoa L\^e}
% \email{khoale.ku.edu}
% \author[D. Nualart]{David Nualart}
% \email{nualart@ku.edu}
\address{David Nualart: Department of Mathematics \\
The University of Kansas \\
Lawrence, Kansas, 66045}
\email{nualart@ku.edu}
\title{Large time asymptotics for the parabolic Anderson model driven by spatially correlated noise}
\begin{abstract}
	 In this paper we study the linear stochastic heat equation, also known as parabolic Anderson model, in multidimension driven by a Gaussian noise which is white in time and it has a correlated spatial covariance. Examples of such covariance include the Riesz kernel in any dimension and the covariance of the fractional Brownian motion with  Hurst parameter $H\in (\frac 14, \frac 12]$ in dimension one. First we establish the existence of a unique mild solution and we derive a Feynman-Kac formula for its moments using a family of independent Brownian bridges and assuming a general integrability condition on the initial data.  In the second part of the paper we  compute Lyapunov exponents, lower and upper exponential growth indices in terms of a variational quantity. The last part of the paper is devoted to study the phase transition property of the Anderson model. 
	 % These results also hold in the framework of a $d$-dimensional spatial variable, assuming that the spatial covariance is a non-negative function and satisfies Dalang's condition.
\end{abstract}
\maketitle
\section{Introduction} % (fold)
In this paper we are interested in the stochastic heat equation  
	\begin{equation}\label{eqn:SHE}
		 \frac {\partial u}{\partial t} =\frac12\Delta u+u \dot{W}\, ,
	\end{equation}
	where $t\ge 0$, $x\in \RR^\ell$ $(\ell\ge1)$ and  $W$ is  a centered Gaussian field, which is white in time and it has a correlated spatial covariance. More precisely, we assume that the noise $W$ is described by a centered  Gaussian family  $W=\{ W(\phi), \phi \in 
	\css(\RR_+\times\RR^\ell)\}$, with covariance
	\begin{equation}\label{eqn:fouriercov}
		\EE [W(\phi)W(\psi)]=\frac1{(2 \pi)^\ell} \int_0^\infty\int_{\RR^\ell} \cff\phi(s,\xi)\overline{\cff\psi(s,\xi)}\mu(d \xi)ds,
	\end{equation}
	where $\mu$ is a tempered  positive measure and $\cff$ denotes the Fourier transform in the spatial variables. If the inverse Fourier transform of $\mu$ is a locally integrable function
	\begin{equation}\label{eqn:gamspec}
		\gamma(x)=\frac1{(2 \pi)^\ell}\int_{\RR^\ell} e^{i \xi  \cdot x}\mu(d \xi)\,,
	\end{equation}
	% {\color{red} Even $\mu$ does not have finite total mass, we can still define the function which is the Fourier transform of it. }
	then, $\gamma$ is positive definite and \eqref{eqn:fouriercov} can be written as
	\begin{equation}\label{eqn:cov}
		\EE [W(\phi)W(\psi)]=\int_0^\infty\iint_{\RR^{2\ell}} \phi(s,x)\psi(s,y)\gamma(x-y)dxdyds\,.
	\end{equation}
	If, in addition, $\gamma$ is nonnegative, the existence of a mild solution for equation (\ref{eqn:SHE}) is equivalent to the following  Dalang's condition (see, for instance, \cite{Dal,DQ}), which plays a fundamental role in this theory,
	\begin{equation} \label{k5}
	\int_{ \RR^\ell}\frac {\mu(d\xi) }{1+ |\xi|^2} <\infty.
	\end{equation}	
	Throughout the paper, we denote by $|\cdot|$ the Euclidean norm in $\RR^\ell$ and by $x\cdot y$ the usual inner product between two vectors $x,y$ in $\RR^\ell$. 
 
  \medskip
 Recently, (see \cite{BJQ,HHLNT}) one has considered the case where $\ell=1$ and  $W$ has the covariance of a fractional Brownian motion with Hurst parameter $H\in (\frac 14, \frac 12)$ in the spatial variable. In this case,  $\gamma$ is not well-defined as a function but a distribution, the right-hand side  of \eqref{eqn:cov}  is thus not well-defined. The spectral measure is, in this case,  $\mu(d\xi)=c_{1,H} |\xi|^{1-2H}d\xi$, where
$c_{1,H}= \Gamma(2H+1)\sin(\pi H) $. 
The main aim of this  paper is to consider a  general class of rough noises,  with spatial dimension $1$, that includes this example. For the solution of equation  (\ref{eqn:SHE}) driven by these rough noises, we  plan to derive Feynman-Kac formulas for the moments,  compute Lyapunov exponents and exponential growth indices and study  the phase transition property of Lyapunov exponents.
For this purpose,  introduce   the following conditions on the spectral measure.
 
\begin{enumerate}[leftmargin=0cm,itemindent=1.4cm,label=(H.1)]
\item\label{H1} 
We assume that $\ell =1$, the spectral measure $\mu$ is absolutely continuous with respect to the Lebesgue measure on $\RR$ with density $f$, that is $\mu(d \xi)=f(\xi)d \xi$, and $f$ satisfies:  
\begin{itemize}
\item [(a)] For all $\xi,\eta$ in $\RR$ and for some constant $\kappa>0$,
\begin{equation}\label{eq:H11}
				f(\xi+\eta)\le  \kappa  (f(\xi)+f(\eta)).
\end{equation}
	\item	 [(b)]  {The spectral density $f$ satisfies
	\begin{equation}\label{eq:H12}
			\int_\RR\frac{f^2(\xi)}{1+|\xi|^{2}}d \xi<\infty\,.
		\end{equation}    }
\end{itemize}		
\end{enumerate}

Hypothesis \ref{H1} is satisfied by the  spectral density $f(\xi)=|\xi|^{1-2H}$, when   $H \in (\frac 14, \frac 12]$, with a constant $\kappa =1$.  On the other hand, (\ref{eq:H12}) in Hypothesis \ref{H1} clearly implies  the integrability condition  (\ref{k5}).

\medskip
On the other hand, our results also hold under the following hypothesis, which we formulate for a general dimension $\ell\ge 1$:
 
\begin{enumerate}[leftmargin=0cm,itemindent=1.4cm,label=(H.2)]
	\item\label{H2}  The inverse Fourier transform of $\mu(d\xi)$ is a nonnegative  locally integrable function $\gamma$ and $\mu$ satisfies Dalang's condition \eqref{k5}. { When $\ell=1$, $\gamma$ can also be the Dirac delta function  {$\delta_0$}, which corresponds to the space-time white noise case. }
\end{enumerate}

Examples of covariance functions satisfying \ref{H2} are the Riesz kernel $\gamma(x)=|x|^{-\eta}$, with $0<\eta <2\wedge \ell$, the space-time white noise in dimension one, where $\gamma =\delta_0$, and the multidimensional fractional Brownian motion, where $\gamma(x)= \prod_{i=1}^\ell H_i (2H_i-1) |x^i|^{2H_i-2}$, assuming $\sum_{i=1}^\ell H_i >\ell-1$ and $H_i >\frac 12$ for $i=1,\dots, \ell$.

\medskip
In this paper we assume that the initial condition $u_0$ is a {tempered measure} satisfying the condition
\begin{equation} \label{eq1}
(p_t *|u_0|)(x) <\infty \mbox{ for all } t>0 \mbox{ and } x\in\RR^\ell\,,
\end{equation}
where $p_t *|u_0|$ denotes the convolution of the  heat kernel $p_t$ and the {total variation of $u_0$, denoted by $|u_0|$}.   
This condition is equivalent to
\begin{equation}  \label{eq14}
\int_{\RR^\ell} e^{-a|x|^2}  {  |u_0| (dx) }<\infty,
\end{equation}
for all $a>0$. 
Using the Wiener chaos expansion, we show in Theorem \ref{thm1}  that, under hypotheses \ref{H1} or \ref{H2},   there exists a unique mild solution to equation (\ref{eqn:SHE}).  
Note that hypothesis (\ref{eq1}) is weaker than condition  $\int_{\RR}(1+|\xi|^{\frac{1}{2}-H})|\mathcal{F}u_0(\xi)|d\xi < \infty $,  imposed in \cite{HHLNT} in the case $\ell=1$ and  $\mu(d\xi)= c_{1,H} |\xi|^{1-2H}$.   
The existence of a unique solution for the stochastic heat  equation under Hypothesis \ref{H2}  with a nonlinear coefficient $\sigma$, when the initial condition $u_0$  a measure   satisfying (\ref{eq1}),  has been proved in \cite{CKK} (the one-dimensional case with space-time white noise was previously treated in \cite{CD}).

On the other hand, in   Proposition \ref{prop1}  we show  a Feynman-Kac formula for the moments of  the solution in   terms of an independent family of Brownian bridges.
	
	\medskip
The second part of our paper is devoted to {computing the speed of propagation of intermittent peaks. The propagation of the farthest high peaks was  first considered by Conus and Khoshnevisan in \cite{CK} for a {one-dimensional} heat equation driven by {a} space-time white noise ($\gamma= {\delta_0}$) with compactly supported initial condition, where it is shown that  there are intermittency fronts that move linearly with time as $\alpha t$. Namely, for any fixed $p \in [2, \infty)$, if $\alpha$ is sufficiently small, then the quantity $\sup_{|x|> \alpha t}\EE (|u(t,x)|^p)$ grows exponentially fast as $t $ tends to $\infty$; whereas the preceding quantity vanishes exponentially fast  if $\alpha$ is sufficiently large. To be more precise, the authors of   \cite{CK}  define for every $\alpha > 0$, 
\begin{equation}\label{linear growth index}
\mathcal{S}(\alpha):= \limsup_{t \to \infty} \frac{1}{t} \sup_{|x|> \alpha t} \log \EE (|u(t,x)|^p)\,,
\end{equation}
and think of $\alpha_L$ as an intermittency lower front if $\mathcal{S}(\alpha) < 0$ for all $\alpha > \alpha_L$, and of $\alpha_U$ as an intermittency upper front if $\mathcal{S}(\alpha)>0$ whenever $\alpha < \alpha_U$. In \cite{CK} it is shown that for each real number $p \geq 2$, $ 0 < \alpha_U \leq \alpha_L < \infty$, and when $p=2$, some  bounds for $\alpha_L$ and $\alpha_U$ are given. In a later work by Chen and Dalang \cite{CD}, using the method of iteration, it is proved that when $p=2$, there exists a critical number $\alpha^* =\frac{\lambda^2}{2}$ such that $\mathcal{S} (\alpha) < 0$ when $\alpha> \alpha^*$ while $\mathcal{S}(\alpha) > 0$ when $\alpha < \alpha^*$ (this property was first conjectured in \cite{CK}).  

In this paper, we will use methods from large { deviations}  to give a more general treatment of this problem. We will show that for a general class of covariance {functions} $\gamma$ and any positive integer $p \geq 2$, the $p$th moment of the solution to (\ref{eqn:SHE}) with certain initial condition develops high peaks which travel at a constant speed when the time is large, and we are able to express the speed, using the variation expression (see (\ref{v5}) below). In the particular case when the initial condition $u_0(x)=1$, we are able to find the precise Lyapunov exponent of the solution. 

Our approach is based on the Feynman-Kac formula for the moments of the solution, usually expressed in terms of independent Brownian motions. As is shown in Theorem \ref{thm:speed}, the propagation of  high peaks largely depends on the initial condition, thus we use the Brownian bridges in the Feynman-Kac moment formula, to isolate the initial condition, which {simplifies its analysis} (see Proposition \ref{prop1}). The procedure is to first get some large deviation results for the long term asymptotic of exponential functionals of Brownian bridges with the regularized covariance function (see, for instance Lemma \ref{lem:Vupper}), then {make} some approximation to treat  general covariance {functions}.  To do this we introduce the approximate covariance
 }
% \begin{enumerate}[leftmargin=0cm,itemindent=1cm,label=(K)]
% 	\item \label{H2c} We assume that
% 	\begin{equation}
% 	\sup_{x\in\RR^\ell}|x\cdot\nabla\gamma_{\epsilon}(x)|<\infty\,,
% 	\end{equation}
% $\gamma_\epsilon$ is defined as
	\begin{equation}\label{eqn:ge}
		\gamma_ \epsilon(x)=\frac1{(2 \pi)^\ell}\int_{\RR^\ell} e^{- \epsilon|\xi|^2} e^{i \xi \cdot x}\mu(d \xi)\,.
	\end{equation}
	For each $\epsilon>0$, the spectral measure of $\gamma_ \epsilon$ is $e^{-\epsilon |\xi|^2}\mu(d \xi)$, which has finite total mass because  $\mu$ is a tempered measure. Thus, $\gamma_ \epsilon$ is a bounded positive definite function. 
	Most of our results are obtained first for $\gamma_ \epsilon$ then for $\gamma$ by the passage $\epsilon\downarrow0$. The same methodology has been used in \cite{ChPh15}.

\medskip
We denote  by $ (\RR^\ell)^n$ (or simply  $\RR^{n \ell  }$)  the set of $n$-ples $x=(x^1, \dots, x^n)$, such that, $x^j\in \RR^\ell$ for each $j=1,\dots, n$. For each $\ell\ge1$, let $\mathcal{F}_{n\ell}$ be the collection of  functions $g\in H^1( \RR^{n\ell  }) : = \{f: \RR^{n\ell} \to \RR\,|\, \int_{\RR^{ n\ell  }} |f(x)|^2 dx < \infty \ \text{and} \ \int_{\RR^{n  \ell  }} |\nabla f(x)|^2 dx < \infty\}$ such that $\|g\|_{L^2(\RR^{ n \ell  } )}=1$. For appropriate Schwartz distributions $\phi\in \mathcal{S}'(\RR^\ell)$, we consider the functional
	\begin{equation}   \label{v5}
		\cee_n(\phi)=\sup_{g\in\mathcal{F}_{n\ell}}\left(\int_{\RR^ {n\ell}} \sum_{1\le j<k\le n} \phi(x^j-x^k) g^2(x)  dx-\frac12\int_{\RR^{n\ell}} |\nabla g(x)|^2dx \right).
	\end{equation}
 We can also represent $\cee_n$ in Fourier mode
	\begin{equation}   \label{v5a}
	  \mathcal{E}_n(\phi)= \sup_{h\in \mathcal{A}_{n\ell}}   \left((2 \pi)^{-\ell} \int_{\RR^{\ell}} \sum_{1\le j<k\le n}    (h*h)(e_{jk} (\xi)) \cff \phi(\xi) d\xi - \frac 12
	  \int_{\RR^{n\ell}} |\xi |^2 |h(\xi)|^2 d\xi \right)\,,
	\end{equation}
where 
\begin{equation}
\mathcal{A}_{n\ell}= \lt\{ h: \RR^{n\ell} \rightarrow \mathbb{C}\,\big|\, \| h\|^2 _{L^2(\RR^{n\ell})} =1,  \int_{\RR^{n\ell}} |\xi |^2 |h(\xi)|^2 d\xi  <\infty \ \text{and} \ \overline{h(\xi)} = h(-\xi)\rt\}\,,
\end{equation}
and for each $1\le j<k\le n$ and $\xi \in \RR^{\ell}$, $e_{jk}(\xi):=(\xi^{1}_{jk},\dots,\xi^{n}_{jk} )$ is the vector in $\RR^{n\ell}$ with all $\xi^{i}_{jk}$, $1\leq i \leq n$, equal to $0$ except $\xi_{jk}^k=\xi$ and $\xi_{jk}^j =-\xi$. For more details on the above expression, we refer the readers to Lemma \ref{lem:e epsilon finite} and Proposition \ref{prop:Efinite}. Using the previous expression, we can define the functional $\cee_n(\gamma)$ assuming $\mu$ is a tempered measure, by 
	\begin{equation}  \label{id:ehat}
	  \mathcal{E}_n(\gamma)= \sup_{h\in \mathcal{A}_{n\ell}}   \left( (2 \pi)^{-\ell} \int_{\RR^{\ell}} \sum_{1\le j<k\le n}    (h*h)(e_{jk} (\xi)) \mu (d\xi) - \frac 12
	  \int_{\RR^{n\ell}} |\xi |^2 |h(\xi)|^2 d\xi \right)\,.
	\end{equation}
	We will see later (Proposition \ref{prop:Efinite}) that $\cee_n(\gamma)$ is a real positive number under hypotheses   \ref{H1} or \ref{H2}.

	\medskip
	 Let $u(t,x)$ be the solution of the stochastic heat equation (\ref{eqn:SHE}) with initial condition $u_0$ and multiplicative noise corresponding to $\dot{W}$. 
 We are interested in the asymptotic behavior as $t\rightarrow \infty$ of the moments
 $\EE \left[\prod_{j=1}^n u(t,x^j)\right]$, where $x^j \in \RR^\ell$, 
for general covariances $\gamma$ satisfying hypotheses  \ref{H1} or \ref{H2} and nonnegative  initial data $u_0$ satisfying (\ref{eq1}). 
 In this direction,  we have obtained the following asymptotic results.
	\begin{theorem}\label{thm:largetimemoment} We assume  either \ref{H1} or \ref{H2}.  Let $u_0$ be a nonnegative function  satisfying (\ref{eq1}). Let $u$ be the solution to the stochastic heat equation \eqref{eqn:SHE} with initial condition at $u_0$. Then for every integer $n\ge2$,
		\begin{align}\label{est:upper}
			\limsup_{t\to\infty}\frac 1t\log\sup_{(x^1,\dots,x^n)\in(\RR^{\ell})^n}\frac{ \EE \left[\prod_{j=1}^n u(t,x^j)\right]}{\prod_{j=1}^n( p_t*u_0)(x^j)}\le\cee_n(\gamma)\,,
		\end{align}
		and for every {$M$ such that the integral in the denominator of (\ref{est:lower}) below is positive and $M'>0$}
		\begin{align}\label{est:lower}
			\liminf_{t\to\infty}\frac 1t\log\inf_{(x^1,\dots,x^n)\in A_{M'}}\frac{ \EE \left[\prod_{j=1}^n u(t,x^j)\right]}{\int_{A_M}\prod_{j=1}^n p_t(y^j)u_0(x^j+y^j)dy }\ge\cee_n(\gamma)\,,
		\end{align}
		where $A_M$ is the set $A_M=\{(y^1,\dots,y^n)\in(\RR^\ell)^n:|y^j-y^k|\le M \text{ for all } 1\le j<k\le n \}${(i.e., a diagonal strip in $(\RR^{\ell})^n$ )} and  $A_{M'}$ is defined analogously.
	\end{theorem}
	{ The choice of $A_M$ and $A_{M'}$ {guarantees} that we can find the exact speed of {propagation of  high peaks} (see Theorem \ref{thm:speed} below and the proof of Lemma \ref{lem:ldev}).}
	It is  worth mentioning that the techniques in proving Theorem \ref{thm:largetimemoment} could be applied to prove the following result.
	\begin{theorem}\label{thm:u01}
		Assume conditions  \ref{H1} or \ref{H2}. Let $u$ be the solution to the stochastic heat equation \eqref{eqn:SHE} with initial condition $u_0(x)=1$ for all $x\in\RR^\ell$. Then for every integer $n\ge2$ and for every $x^1,\dots,x^n\in\RR^\ell$, we have
		\begin{align}
			\lim_{t\to\infty}\frac 1t\log \EE \left[\prod_{j=1}^n u(t,x^j)\right]=\cee_n(\gamma)\,.
		\end{align}
	\end{theorem}
	This theorem is a straightforward application of Proposition \ref{prop:expBM}. 
	% Comparing with the hypotheses in Theorem \ref{thm:largetimemoment}, condition \ref{H2c} is not required in Theorem \ref{thm:u01}. This is because Theorem \ref{thm:largetimemoment} is obtained by considering large time asymptotics of functionals of Brownian bridges, while Theorem \ref{thm:u01} is concerned with large time asymptotics  of functionals of Brownian motions. For further details, we refer to Propositions \ref{prop:expgam} and \ref{prop:expBM}.  
	
 \medskip
	Inequalities \eqref{est:upper} and \eqref{est:lower} can  be used to show the  exact speed of propagation of intermittent peaks.  We recall the definition of lower and upper exponential growth indices  from \cite{CK}
	\begin{align*}
		\lambda_*(n)=\sup\left\{\alpha>0:\liminf_{t\to\infty}\frac 1t \sup_{|x|\ge \alpha t}\log\EE |u(t,x)|^n>0\right \}\,,
	\end{align*}
	\begin{align*}
		\lambda^*(n)=\inf\left\{\alpha>0:\limsup_{t\to\infty}\frac 1t \sup_{|x|\ge \alpha t}\log\EE  |u(t,x)|^n<0\right \}\,.
	\end{align*}
	Actually in \cite{CK} (and also in other works \cite{CD,CKK}), the above definitions are only for one dimension. However, the following result holds in any dimensions. 
	\begin{theorem}\label{thm:speed}  
	We assume  either  \ref{H1} or \ref{H2}.  Let $u_0$ be a nontrivial and nonnegative function  satisfying (\ref{eq1}). Then for each $n\ge2$, the following estimates hold
	\begin{equation}\label{est:llow}
	  	\lambda_*(n)\ge \sqrt{\frac{2\cee_n(\gamma)}{n}}
	\end{equation}
	and	\begin{equation}\label{est:lup}
	  	\lambda^*(n)\le \inf_{\beta}\left( \frac {\beta}{2}+\frac{\cee_n(\gamma)}{n \beta}\right)
	\end{equation}     
	where the infimum is taken over all $\beta>0$ such that $\int_{\RR^\ell} e^{\beta|y|}u_0(y)dy<\infty$.

	In addition, if $u_0$ satisfies $u_0(y)\ge Ce^{-\beta|y|}$ for some constants $\beta>0$ and $C>0$, then
	\begin{equation}\label{est:llowbeta}
		\lambda_*(n)\ge 
		\left\{
		\begin{array}{lll}
			\frac \beta2+\frac{\cee_n(\gamma)}{n \beta} & \text{ if } & \beta<\sqrt{\frac{2\cee_n(\gamma)}n} \\
			\sqrt{\frac{2\cee_n(\gamma)}n} & \text{ if } & \beta\ge\sqrt{\frac{2\cee_n(\gamma)}n}
			 \end{array}
		\right.\,.
	\end{equation}
	% \begin{equation}\label{est:llowbeta}
	% {\color {red}	\lambda_*(n)\ge  \inf_{\beta' \leq \beta} \left( \frac {\beta'}{2}+\frac{\cee_n(\gamma)}{n \beta'}\right)\,.    }
	% \end{equation}
	% We denote by $\mathcal{A}(u_0)$ the set
	% 	\begin{equation}
	% 		\mathcal{A}(u_0):=\{\beta>0: \int_\RR e^{\beta|y|}u_0(y)dy<\infty \}.
	% 	\end{equation} 
	% 	Then for every $n\ge2$, we have
	% 	\begin{equation}\label{eqn:speedn}
	% 		\lambda_*(n)=\lambda^*(n)=\inf_{\beta\in\mathcal{A}(u_0)}\left( \frac \beta2+\frac{\cee_n(\gamma)}{n \beta}\right)\,.
	% 	\end{equation}
	% 	If $\mathcal{A}(u_0)=\emptyset$, then $\lambda_*(n)=\lambda^*(n)=\infty$.
	\end{theorem}
	We discuss a few consequences of the previous theorem.
	\begin{enumerate}[leftmargin=0cm,itemindent=1cm,label=\bf{\arabic*.}]
		\item If $u_0$ is nontrivial and supported on a compact set, then
		\begin{equation}\label{id:speed1}
			\lambda^*(n)=\lambda_*(n)=\sqrt{\frac{2\cee_n(\gamma)}{n}}\,.
		\end{equation}
		Indeed, since $u_0$ has compact support, the infimum in \eqref{est:lup} can be taken over all $\beta>0$. Hence both estimates \eqref{est:llow} and \eqref{est:lup} yield the same bound. 
		\item If $u_0$ satisfies  $c_1e^{-\beta|y|}\le u_0(y)\le c_2e^{-\beta|y|} $ for some positive constants $c_1,c_2$ and $\beta$, then
		\begin{equation}\label{id:speed2}
			\lambda^*(n)=\lambda_*(n)= 
			\left\{
		\begin{array}{lll}
			\frac \beta2+\frac{\cee_n(\gamma)}{n \beta} & \text{ if } & \beta<\sqrt{\frac{2\cee_n(\gamma)}n} \\
			\sqrt{\frac{2\cee_n(\gamma)}n} & \text{ if } & \beta\ge\sqrt{\frac{2\cee_n(\gamma)}n}
			 \end{array}
		\right.\,.
		\end{equation}
		This clearly follows from \eqref{est:lup} and \eqref{est:llowbeta}. 
	\end{enumerate}

	In the case of space-time white noise, i.e. $\mu$ is the Lebesgue measure on $\RR$ and $\gamma$ is the Dirac mass at 0, $\cee_n(\delta)$ is computed explicitly in \cite{Ch15}. Namely,
	\begin{equation}\label{id:ceedelta}
		\cee_n(\delta)=\frac{n(n^2-1)}{24}\,.
	\end{equation}
	% The right hand sides of \eqref{id:speed1} and \eqref{id:speed2} become 
	% \begin{equation*}
	% 	\sqrt{\frac{n^2-1}{12}}\mbox{ and }\frac{\beta}2+\frac{n^2-1}{24 \beta}
	% \end{equation*}
	% respectively. 
	When $n=2$, we recover from \eqref{id:speed1} and \eqref{id:speed2} the result of Chen and Dalang \cite{CD}.
 
% \begin{remark}  \label{rem7}
% Theorems  \ref{thm:largetimemoment}, \ref{thm:u01}  and \ref{thm:speed} also hold if  hypotheses \ref{H1a},\ref{H1b} and \ref{H2c} are replaced by Hypothesis \ref{H2}, 
% following the same arguments as we have used in  the paper, but with much less technical difficulties. We leave the reader to complete the details.  
% \end{remark}

\medskip
	{For some covariances, if we tune the magnitude of the noise to be small enough, the $\mathcal {E}_n(\gamma)$ above can actually be zero. This is called phase transition and is studied in the last part of the paper.} %The last part of the paper studies phase transition property of intermittency. 
To be more precise, for each fixed $\lambda>0$, let $u_ \lambda$ be the solution to the stochastic heat equation
		\begin{equation}\label{eqn:ulambda}
			 \frac {\partial u}{\partial t} =\frac12 \Delta u+\sqrt\lambda u \dot{W}\, ,
		\end{equation}
	with initial condition $u_ \lambda(0,x)=1$ for all $x\in\RR^\ell$.  We describe some notions of phase transition. The following definition is based on \cite{CK,CKK}.
	\begin{definition}\label{def:weakphase}
		For each integer $n\ge2$, we say that \textit{weak phase transition} occurs at order $n$, if there exist critical values $\overline\lambda^c_n\ge\underline\lambda^c_n>0$ such that
		\begin{equation*}
			\limsup_{t\to\infty}\frac1t \log\EE u^n_ \lambda(t,x) =0
		\end{equation*} 
		whenever $\lambda<\underline \lambda^c_n$ and
		\begin{equation*}
			\liminf_{t\to\infty}\frac1t \log\EE u^n_ \lambda(t,x) >0
		\end{equation*}
		whenever $\lambda>\overline \lambda^c_n$.
	\end{definition}
	\begin{definition}\label{def:phase}
		For each integer $n\ge2$, we say that \textit{(strong) phase transition} occurs at order $n$, if there exists a critical value $\lambda^c_n>0$ such that
		\begin{equation*}
			\limsup_{t\to\infty}\frac1t \log\EE u^n_ \lambda(t,x) =0
		\end{equation*}
		whenever $\lambda< \lambda^c_n$ and
		\begin{equation*}
			\liminf_{t\to\infty}\frac1t \log\EE u^n_ \lambda(t,x) >0
		\end{equation*}
		whenever $\lambda>\lambda^c_n$.
	\end{definition}
	Phase transition is studied as early as \cite{N}. Several stochastic processes possess phase transition property, see for instance \cite{CKK,FN}.

	{We will see that when the magnitude of certain noise is small enough, the growth index treated in Theorem \ref{thm:speed} is $0$, that is, there is no exponential high peaks propagating.} We collect a few useful observations in the following result. 
	\begin{proposition}\label{prop:phase}
		Under assumptions  \ref{H1} or \ref{H2}, the following statements hold:
		\begin{enumerate}[(i)]
		 	\item Weak phase transition at an order implies (strong) phase transition at the same order.
		 	\item If (strong) phase transition happens at an order then (strong) phase transition happens at all orders.
		 	\item If phase transition happens then
				\begin{equation}
					\lambda^c_2\le\frac{\ell (2 \pi)^\ell }{4  \sup_{s>0} \int_{\RR^\ell} se^{-s|\xi|^2}\mu(d \xi) }\,.
				\end{equation}
				In other words, the function $\lambda\mapsto\cee_2(\lambda \gamma)$ is nontrivial and intermittency always happens provided $\lambda$ is sufficiently large. 
			\item If
			\begin{equation}
				\sup_{s>0}\int_{\RR ^\ell} s e^{-s|\xi|^2}\mu(d \xi)=\infty\,,
			\end{equation}
			then there is no phase transition. 
		 \end{enumerate} 
	\end{proposition}
	Because of the previous result, it makes sense to say ``transition occurs'' without specifying the order. 
	\begin{theorem}\label{thm:front1}
		Assume  condition \ref{H2}. Phase transition occurs if and only if
		\begin{equation}\label{cond:front1}
			\int_{\RR^\ell}\frac{\mu(d \xi)}{|\xi|^2}<\infty\,.
		\end{equation}
		In addition, if phase transition occurs, then for all $n\ge 2$, we have
			\begin{equation}\label{est:lcabove}
				\lambda_n^c\le  \frac{(2\pi)^\ell e}{ 2\int_{\RR^\ell}\frac{\mu(d \xi)}{|\xi|^2}}\,.
			\end{equation}
		% \begin{enumerate}[(i)]
		% 	\item If the integral
		% 	\begin{equation*}
		% 	 	\int_{\RR^\ell}\frac{\mu(d \xi)}{|\xi|^2})
		% 	\end{equation*} is finite, then there exists a critical value $0<\lambda_n^c<\infty$ such that $\cee_n(\lambda \gamma)=0$ for all $\lambda\in (0,\lambda_n^c]$ and $\cee_n(\lambda \gamma)>0$ for all $\lambda>\lambda_n^c$. In addition, 
		% 	\begin{equation}\label{est:lcabove}
		% 		\lambda_n^c\le  \frac{e}{ \int_{\RR}\frac{\mu(d \xi)}{|\xi|^2}}\,.
		% 	\end{equation}
		% 	\item If the integral
		% 	\begin{equation*}
		% 	 	\int_{\RR^\ell}\frac{\gamma(z)}{|z|^{\ell-2}}dz\quad (\mbox{or equivalently } \int_{\RR^\ell}\frac{\mu(d \xi)}{|\xi|^2})
		% 	\end{equation*}  is infinite, then $\cee_n(\lambda \gamma)>0$ for all $\lambda>0$ 
		% \end{enumerate}
	\end{theorem}
	Under hypothesis \ref{H1}, the picture is less complete. 
	\begin{theorem}\label{thm:front2}
		Assuming  condition \ref{H1}, if 
		\begin{equation}
			\int_\RR \frac{f(\xi)+f^2(\xi)}{|\xi|^2}d \xi<\infty\,,
		\end{equation}
		then phase transition happens. 
	\end{theorem}

	The paper is organized as follows. Section \ref{sec:preliminaries} quickly recalls some elements of stochastic calculus. Section \ref{sec:existence} discusses existence and uniqueness of the solution to \eqref{eqn:SHE}. In Section \ref{sec:FKformula}, we obtain a Feynman-Kac formula for the moments of the solution, which is based on Brownian bridges. In Section \ref{sec:largedev}, we study large time asymptotics of exponential functionals of Brownian bridges and Brownian motions. The proof of Theorem \ref{thm:largetimemoment} is provided in Section \ref{sec:proof12}. Section \ref{sec:propagation} discusses results about exponential growth indices. In particular, a proof of Theorem \ref{thm:speed} is given. Section \ref{sec:phase} discusses results about phase transition. We present in that section the proofs of Proposition \ref{prop:phase}, Theorem \ref{thm:front1} and Theorem \ref{thm:front2}. Section \ref{sec:app} is an appendix which contains some technical results.

\section{Preliminaries}
\label{sec:preliminaries}

  The space of   Schwartz functions is
denoted by $\mathcal{S}(\RR^\ell)$.   The Fourier
transform of a function $u \in \mathcal{S}(\RR^\ell)$ is defined with the normalization
\[ \mathcal{F}u ( \xi)  = \int_{\mathbb{R}^\ell} e^{- i
   \xi \cdot x } u ( x) d x, \]
so that the inverse Fourier transform is given by $\mathcal{F}^{- 1} u ( \xi)
= ( 2 \pi)^{- \ell} \mathcal{F}u ( - \xi)$.

% We describe the noise $W$   by a zero-mean Gaussian family $\{W(\phi) ,\, \phi\in
% \mathcal{S}(\RR_+\times \RR)\}$ defined on a complete probability space
% $(\Omega,\cf,\PP)$, whose covariance structure
% is given by
% \begin{equation}\label{eq:cov1}
% \EE\lc W(\phi) \, W(\psi) \rc
% = \int_0^\infty\int_\RR
% \cf\phi(s,\xi) \, \overline{\cf\psi(s,\xi)} \, \mu(d\xi)\, ds,
% \end{equation}
% where   $\cf $ denotes the Fourier transform in the space variable and $\mu$ is a measure satisfying conditions (B) or (C1) and (C2).

\subsection{Stochastic integration with respect to $W$}
We can interpret  $W$   as a Brownian motion with values in an infinite dimensional Hilbert space. In this context,  the stochastic integration theory with respect to $W$ can be handled by classical theories.
  % (see e.g %\cites{Dal,DQ,DPZ}). 
We briefly recall the main   features of this theory.

We denote by  $\HH_0$ the Hilbert space defined as the closure of  $\mathcal{S}(\RR^\ell)$ under the inner product
  \begin{equation}\label{HH0}
  \langle g, h \rangle_{ \HH_0}=\frac 1{(2\pi)^\ell}  \int_{\RR^\ell}\mathcal{F}g(\xi)\overline{\mathcal{F}h(\xi)} \mu(d\xi).
  \end{equation}
  Then the Gaussian family $W$ can be extended to an  {\it isonormal Gaussian} process $\{W(\phi), \phi \in L^2(\RR_+, \HH_0)\}$ parametrized by the Hilbert space $\HH:=L^2(\RR_+, \HH_0)$.
For any $t\ge0$, let $\mathcal{F}_{t}$ be the $\sigma$-algebra generated by $W$ up to time $t$. An elementary process $g$ is an adapted  step process with values in $\HH_0$  given by
\begin{equation*}
g(s)
=
\sum_{i=1}^{n} \sum_{j=1}^m X_{i,j} \, \mathbf{1}_{(a_{i},b_{i}]}(s) \,  \phi_j,
\end{equation*}
where $n$ and $m$ are finite positive integers, $0 \le a_{1}<b_{1}<\cdots<a_{n}<b_{n}<\infty$, $\phi_j\in \HH_0$    and $X_{i,j}$ are $\cf_{a_{i}}$-measurable random variables for $i=1,\ldots,n$, $j=1\dots, m$. The integral of  such a process with respect to $W$ is defined as
\[
\int_0^\infty\int_{\mathbb{R}^\ell}g(s,x) \, W(ds,dx)
=\sum_{i=1}^{n} \sum_{j=1}^m X_{i,j} \, W\lp  \mathbf{1}_{(a_{i},b_{i}]} \otimes \phi_j \rp.
\]
We use here the notation $g(s,x)$, although $g(s, \cdot)$ is not necessarily a function, but an element of the Hilbert space $\HH_0$.
Let $\laa$ be the space of  $\HH_0$-valued predictable processes $g$  such that
 $\EE\|g\|_{\HH}^{2}<\infty$. Then, it is not difficult so  show that 
the space of elementary processes  is dense in $\laa$, and for $g\in\laa$, the stochastic integral $\int_0^\infty\int_{\mathbb{R}^\ell}g(s,x) \, W(ds,dx)$ is defined as the $L^{2}(\oom)$-limit of Riemann sums along elementary processes approximating $g$. Moreover, we have
\begin{equation}\label{int isometry}
\EE  \lp \int_0^\infty\int_{\mathbb{R}^\ell}g(s,x) \, W(ds,dx) \rp^{2} 
=
\EE  \|g\|_{\HH}^{2}.
\end{equation}

\subsection{Elements of Malliavin calculus}

We recall   that the Gaussian family $W$ can be extended to $\HH$ and this produces an isonormal Gaussian process. We refer to~\cite{Nua}
for a detailed account of the Malliavin calculus with respect to 
Gaussian processes.  The  smooth and cylindrical
random variables $F$ are of the form
\begin{equation*}
F=f(W(\phi_1),\dots,W(\phi_n))\,,
\end{equation*}
with $\phi_i \in \HH$, $f \in C^{\infty}_p (\RR^n)$ (namely,  $f$ and all
its partial derivatives have polynomial growth). For this kind of random variable, the derivative operator $D$ in the sense of Malliavin calculus is the
$\HH$-valued random variable defined by
\begin{equation*}
DF=\sum_{j=1}^n\frac{\partial f}{\partial
x_j}(W(\phi_1),\dots,W(\phi_n))\phi_j\,.
\end{equation*}
The operator $D$ is closable from $L^2(\Omega)$ into $L^2(\Omega;
\HH)$  and we define the Sobolev space $\mathbb{D}^{1,2}$ as
the closure of the space of smooth and cylindrical random variables
under the norm
\[
\|DF\|_{1,2}=\sqrt{\EE F^2+\EE \|DF\|^2_{\HH}  }\,.
\]
We denote by $\delta$ the adjoint of the derivative operator (or divergence) given
by the duality formula
\begin{equation}\label{dual}
\EE  \lc \delta (u)F \rc =\EE  \lc \langle DF,u
\rangle_{\HH}\rc ,
\end{equation}
for any $F \in \mathbb{D}^{1,2}$ and any element $u \in L^2(\Omega;
\HH)$ in the domain of $\delta$.
As in the one-dimensional case, it can be proved that the space $\laa$ is included in the domain of $\delta$ and for any $u\in \delta$, $\delta(u)$ coincides with the stochastic integral defined above, that is,
\[
\delta(u)=\int_0^\infty\int_{\mathbb{R}^\ell}u(s,x) \, W(ds,dx).
\]

For any integer $n\ge 0$ we denote by $\mathbf{H}_n$ the $n$th Wiener chaos of $W$. We recall that $\mathbf{H}_0$ is simply  $\RR$ and for $n\ge 1$, $\mathbf {H}_n$ is the closed linear subspace of $L^2(\Omega)$ generated by the random variables $\{ H_n(W(\phi)),\phi \in \HH, \|\phi\|_{\HH}=1 \}$, where $H_n$ is the $n$th Hermite polynomial.
For any $n\ge 1$, we denote by $\HH^{\otimes n}$ (resp. $\HH^{\odot n}$) the $n$th tensor product (resp. the $n$th  symmetric tensor product) of $\HH$. Then, the mapping $I_n(\phi^{\otimes n})= H_n(W(\phi))$ can be extended to a linear isometry between    $\HH^{\odot n}$ (equipped with the modified norm $\sqrt{n!}\| \cdot\|_{\HH^{\otimes n}}$) and $\mathbf{H}_n$.

Consider now a random variable $F\in L^2(\Omega)$ which is measurable with respect to the $\sigma$-field  $\cf$ generated by $W$. This random variable can be expressed as
\begin{equation}\label{eq:chaos-dcp}
F= \EE  F + \sum_{n=1} ^\infty I_n(f_n),
\end{equation}
where the series converges in $L^2(\Omega)$, and the elements $f_n \in \HH ^{\odot n}$, $n\ge 1$, are determined by $F$.  This identity is called the Wiener-chaos expansion of $F$.

The Skorohod integral (or divergence) of a random field $u$ can be
computed by  using the Wiener chaos expansion. More precisely,
suppose that $u=\{u(t), t\ge 0\}$ is an $\HH_0$-valued   adapted process such that for all $t\ge 0$,
$\EE \|u(t)\|^2_{\HH_0}<\infty$. 
Then, for each $t\ge 0$, the $\HH_0$-valued random variable $u(t)$ has a Wiener chaos expansion of the form
\begin{equation}  \label{exp1}
u(t)= \EE \lc u(t) \rc + \sum_{n=1}^\infty I_n (f_n(t)),
\end{equation}
where the kernels $f_n(t)$
in the expansion (\ref{exp1}) are  symmetric functions in $\HH
^{\otimes n}$. In
this situation, $u$ belongs to  $\laa$, which is equivalent to say that $u$ belongs to the domain of the divergence operator,   if and only if
the following series converges in $L^2(\Omega)$
\begin{equation}\label{eq:delta-u-chaos}
\delta(u)= \int_0 ^\infty \int_{\RR^\ell}  u(t,x) \, d W(t,x)
= W(\EE u) + \sum_{n=1}^\infty I_{n+1} (\widetilde{f}_n(\cdot,t)),
\end{equation}
where $\widetilde{f}_n$ denotes the symmetrization of $f_n$ in all its $n+1$ variables.  

\subsection{Brownian bridges} % (fold)
\label{sub:brownian_bridges}
	Throughout the paper, we denote by $B_{a,b}=\{ B_{a,b}(s),s\in[a,b]\}$ a Brownian bridge in $\RR^\ell$ which starts and ends at 0. More precisely,    $B_{a,b}(s)=(B_{a,b}^1(s),\dots,B_{a,b}^\ell(s))$ where $\{B_{a,b}^j(s),s\in[a,b]\}$, $j=1,\dots,\ell$, are independent centered Gaussian processes in $\RR$ with covariance function
	\begin{equation*}
		\EE[B_{a,b}^j(s)B_{a,b}^j(r)]=(r-a)(1-\frac{s-a}{b-a})
	\end{equation*}
	for all $a\le r\le s\le b$.
	\begin{lemma}\label{lem:BB decomp}
		Consider a Brownian bridge $B_{a,b}$ in $\RR^\ell$ over the time interval $[a,b]$ such that $B_{a,b}(a)=B_{a,b}(b) =0$. For every fixed $c\in[a,b]$, we have the following decomposition 
		 \begin{equation}\label{eqn:bbdecomp}
		 B_{a,b}(s) = \frac{b-s}{b-c}B_{a,b}(c) + \widetilde{B}_{c,b}(s)
		 \end{equation}
		 for all $s\in[c,b]$, where $ \widetilde{B}_{c,b}$ is another Brownian bridge over $[c,b]$ independent of $\{B_{a,b}(s),s\in[a,c]\}$.
 	\end{lemma}
 \begin{proof}
 Let $a \leq r \leq c$ and $c \leq s \leq b$. Define $\widetilde{B}_{c,b}(s) := B_{a,b}(s) - \frac{b-s}{b-c}B_{a,b}(c)$. Then direct calculations show that $\{\widetilde B(s),s\in[c,b]\}$ has the law of a Brownian bridge over $[c,b]$ and $\EE [\widetilde{B}_{c,b}(s) B_{a,b}(r)]= 0$. These facts imply the result. 
 \end{proof}

%{\color{blue} I used capital letter for the processes below and also brackets instead of parenthesis}

 For every $t>0$ and fixed $x,y\in\RR^\ell$, the process $Z= \{B_{0,t}(s)+\frac st y+(1-\frac st)x; 0\le s\le t \} $ is a Brownian bridge which starts at $x$ at time $0$ and arrives at $y$ at time $t$. Furthermore, there exists a Brownian motion $\{ B(s);s\ge0\}$ starting at 0 so that $Z$ is a solution of the stochastic differential equation
 \begin{equation*}
 	dZ(s)=dB(s)+\frac{y-Z(s)}{t-s}ds\quad\mbox{ for } 0\le s<t
 \end{equation*}
 with boundary values $Z(0)=x$ and $Z(t)=y$.  {Using}  Girsanov theorem, it can be shown that away from the terminal time $t$, the law of $Z$ admits a density with respect to $B$. More precisely, for every $\lambda\in(0,1)$ and every bounded 
{measurable} function $F$ on  
{$C([0,\lambda t]; \mathbb{R}^d)$} we have
 \begin{multline}\label{id:density}
	\EE \lt[F(\{Z(s);0\le s\le \lambda t\}) \rt]
		\\=(1- \lambda)^{-\frac\ell2} \EE \lt[\exp\left\{\frac{|y-x|^2}{2t}-\frac{|y-x-B(\lambda t)|^2}{2t(1- \lambda)} \right\}F(\{x+ B(s);0\le s\le \lambda t\})\rt].
 \end{multline}
 %{\color{red} I think the original expression with $x$ is correct. Kle agree}
 For a proof of this result, we refer to \cite {Nakao}*{Lemma 3.1}.

\section{Existence and  uniqueness of a solution}\label{sec:existence}
We state the definition of the solution to equation  (\ref{eqn:SHE}) using the stochastic integral introduced in the previous section.  

\begin{definition}\label{def-sol-sigma}
Let $u=\{u(t,x),  t\ge 0, x \in \mathbb{R}^\ell\}$ be a real-valued predictable stochastic process  such that for all $t \ge 0$ and $x\in\RR^\ell$ the process $\{p_{t-s}(x-y)u(s,y) \mathbf{1}_{[0,t]}(s), 0 \leq s \leq t, y \in \mathbb{R}^\ell\}$ is an element of $\laa$.  Assume that $u_0$ satisfies  (\ref{eq1}). We say that $u$ is a mild solution of  (\ref{eqn:SHE}) if for all $t \in [0,T]$ and $x\in \mathbb{R}^\ell$ we have
\begin{equation}\label{eq:mild-formulation sigma}
u(t,x)=(p_t * u_0)(x) + \int_0^t \int_{\mathbb{R}^\ell}p_{t-s}(x-y)u(s,y) W(ds,dy) \quad a.s.
\end{equation}
\end{definition}

\begin{theorem}  \label{thm1}
Let $u_0$ be a function satisfying condition (\ref{eq1}).   Suppose that the spectral  measure $\mu$ satisfies  hypotheses \ref{H1} or \ref{H2}.    Then there exists a unique solution to equation (\ref{eqn:SHE}).
\end{theorem}	

\begin{proof}
% To simplify the presentation we assume $\ell=1$. 
It can be proved that a solution $u(t,x)$  admits the chaos expansion 
\begin{equation}  \label{k6}
 u(t,x) = \sum_{n=0}^{\infty} I_n(f_n(\cdot,t,x))\,,
  \end{equation}
  where 
  \begin{multline}\label{eq:expression-fn}
f_n(s_1,x^1,\dots,s_n,x^n,t,x)\\
=\frac{1}{n!}p_{t-s_{\si(n)}}(x-x^{\si(n)})\cdots p_{s_{\si(2)}-s_{\si(1)}}(x^{\si(2)}-x^{\si(1)})
(p_{s_{\si(1)}}*u_0)(x^{\si(1)})\,,
\end{multline}
and  $\si$ denotes the permutation of $\{1,2,\dots,n\}$ such that $0<s_{\si(1)}<\cdots<s_{\si(n)}<t$
(see, for instance,  formula (4.4) in \cite{HN} or  formula (3.3) in \cite{HHNT}).

 To prove the existence and uniqueness of the solution, it suffices to show the convergence of the chaos expansion  (\ref{k6}) in $L^2(\Omega)$, that is, 
 \begin{equation} \label{k7}
 \sum_{n=0}^{\infty} n! \|f_n(\cdot,t,x)\|^2_{{\HH}^{\otimes n}}< \infty\,.
 \end{equation}
As in \cite{HHLNT}, we have, with the convention $s_{\sigma(n+1)} =t$,
\begin{eqnarray*}
&&n! \|f_n(\cdot,t,x)\|^2_{{\HH}^{\otimes n}}\\
&=&\frac{C^n}{n!} \int_{[0,t]^n} \int_{\RR^{n\ell}} \Big |\int_{\RR^{\ell}} \prod_{j=1}^n e^{-\frac{1}{2} (s_{\sigma(j+1)} - s_{\sigma(j)}) |\xi^{\sigma(j)} + \cdots + \xi^{\sigma(1)}-\zeta|^2}   e^{-ix\cdot(\xi^{\sigma(n)} + \cdots + \xi^{\sigma(1)} - \zeta)} \\
&&\times \mathcal{F}u_0(\zeta) e^{-\frac{s_{\sigma(1)} |\zeta|^2}{2}} d\zeta\Big|^2 \mu(d\xi) ds\\
&=& \frac{C^n}{n!} \int_{[0,t]^n} \int_{\RR^{n\ell}} \Big |\int_{\RR^{\ell}} \prod_{j=1}^n e^{-\frac{1}{2} (s_{\sigma(j+1)} - s_{\sigma(j)}) |\xi^{\sigma(j)} + \cdots + \xi^{\sigma(1)}-\zeta|^2}   e^{ix\cdot\zeta} \mathcal{F}u_0(\zeta) e^{-\frac{s_{\sigma(1)} |\zeta|^2}{2}} d\zeta\Big|^2  \mu(d\xi) ds\\
&=&\frac{C^n}{n!} \int_{[0,t]^n} \int_{\RR^{n\ell}} \Big |\int_{\RR^{\ell}} e^{-\frac{t}{2} |\zeta|^2- \sum_{j=1}^n (s_{\sigma(j+1)}-s_{\sigma(j)}) \zeta \cdot (\xi^{\sigma(j)}+\cdots + \xi^{\sigma(1)}) - \frac 12 \sum_{j=1}^n(s_{\sigma(j+1)} - s_{\sigma(j)}) |\xi^{\sigma(j)} + \cdots + \xi^{\sigma(1)}|^2}    \\
&&\times   e^{ix\cdot\zeta} \mathcal{F}u_0(\zeta)  d\zeta\Big|^2  \mu(d\xi)ds\,,
\end{eqnarray*}
for some constant $C>0$, where $\mu(d\xi) =\prod_{j=1}^n \mu(d\xi^j)$ and $ds=ds_1 \cdots ds_n$.
Then by Plancherel's theorem, and the fact that 
\begin{equation}
\mathcal{F}(e^{-a|x|^2-b\cdot x})(y)=\left(\frac{\pi}{a}\right)^{\frac{\ell}{2}} e^{\frac{(b+iy)^2}{4a}}\,,
\end{equation}
for every $y\in\RR^\ell$, we have 
\begin{eqnarray*}
&&n! \|f_n(\cdot,t,x)\|^2_{{\HH}^{\otimes n}}\\
&=& \frac{C^n}{n!} \int_{[0,t]^n} \int_{\RR^{n\ell}} \Big |\int_{\RR^{\ell}} \left(\frac{2\pi}{t}\right)^{\frac{\ell}{2}} \exp \Big(\frac{1}{2t} \big( -\sum_{j=1}^n (s_{\sigma(j+1)} -s_{\sigma(j)}) (\xi^{\sigma(j)}+\cdots+\xi^{\sigma(1)} ) + i y \big)^2 \\
&&-\frac{1}{2} \sum_{j=1}^n(s_{\sigma(j+1)} - s_{\sigma(j)}) |\xi^{\sigma(j)} + \cdots + \xi^{\sigma(1)}|^2 \Big) u_0(y + x)  dy\Big|^2  \mu(d\xi) ds\\
&\leq& \frac{C^n}{n!} \int_{[0,t]^n} \int_{\RR^{n\ell}} \Big |\int_{\RR^{\ell}} \left(\frac{2\pi}{t}\right)^{\frac{\ell}{2}} \exp \Big(\frac{1}{2t} \big|\sum_{j=1}^n (s_{\sigma(j+1)} -s_{\sigma(j)}) (\xi^{\sigma(j)}+\cdots+\xi^{\sigma(1)} ) \big|^2 - \frac{|y|^2}{2t} \\
&&-\frac{1}{2} \sum_{j=1}^n(s_{\sigma(j+1)} - s_{\sigma(j)}) |\xi^{\sigma(j)} + \cdots + \xi^{\sigma(1)}|^2 \Big) |u_0|(y + x)  dy\Big|^2   \mu(d\xi)ds.
\end{eqnarray*}
As a consequence, we can write 
\begin{eqnarray*}
&&n! \|f_n(\cdot,t,x)\|^2_{{\HH}^{\otimes n}}\\
&\leq&\frac{C^n(2\pi)^{2\ell}}{n!}   (p_t*|u_0|(x))^2 \int_{[0,t]^n} \int_{\RR^{n\ell}} \exp \Big(\frac{1}{t} \big|\sum_{j=1}^n (s_{\sigma(j+1)} -s_{\sigma(j)}) (\xi^{\sigma(j)}+\cdots+\xi^{\sigma(1)} ) \big|^2 \\
&&- \sum_{j=1}^n(s_{\sigma(j+1)} - s_{\sigma(j)}) |\xi^{\sigma(j)} + \cdots + \xi^{\sigma(1)}|^2 \Big)   \mu(d\xi) ds\\
&=&{C^n (2\pi)^{2\ell}} (p_t*|u_0|(x))^2  \int_{[0,t]_<^n} \int_{\RR^{n\ell}} \exp \Big(\frac{1}{t} \big|\sum_{j=1}^n (s_{j+1} -s_{j}) (\xi^{j}+\cdots+\xi^{1} ) \big|^2 \\
&&- \sum_{j=1}^n(s_{j+1} - s_{j}) |\xi^{j} + \cdots + \xi^{1}|^2 \Big)   \mu(d\xi) ds\,,
\end{eqnarray*}
where 
\begin{equation}  \label{k8}
 		[0,t]^n_< := \{(t_1,\dots,t_n): 0 < t_1 < \cdots < t_n < t\}.
 	\end{equation}
Finally,  the convergence  (\ref{k7}) follows from Lemma \ref{Lemma:Conv chaos} {in the appendix}. 
% ,taking  $\beta=0$ and  $D=C$.
\end{proof}

\section{Feynman-Kac formula for the moments in terms of Brownian bridges} % (fold)
\label{sec:FKformula}
	We collect some auxiliary results which are needed to prove our main results. 
	\begin{lemma}\label{lem:gamma}
		Let $\gamma:\RR^\ell\to\RR$ be a {bounded function whose Fourier transform (spectral measure $\mu$) is a nonnegative tempered measure.} Let $G=(G^1,\dots,G^n)\in(\RR^\ell)^n$ be a centered Gaussian process indexed by $[0,t]$. 
		For every function  $y=(y^{jk})_{1\le j<k\le n} :[0,t]\to(\RR^\ell)^{n(n-1)/2}$ and real number $a \in \RR$, we have
		\begin{equation}\label{est:gamy}
			\EE \exp\left\{\int_0^t \sum_{1\le j<k\le n}  a \gamma( G_s^j- G_s^k+ y_s^{jk})ds\right\}\le \EE\exp\left\{\int_0^t\sum_{1\le j<k\le n}  |a| \gamma( G_s^j- G_s^k)ds\right\}.
		\end{equation}
	\end{lemma}
	\begin{proof} 
		It suffices to show  that for every $d\ge1$,
		\begin{equation}\label{est:gaman}
			\EE\left[\int_0^t \sum_{1\le j<k\le n} a \gamma(G_s^j-G_s^k+y_s^{jk})ds \right]^d\le \EE\left[\int_0^t \sum_{1\le j<k\le n} |a|\gamma(G_s^j-G_s^k)ds \right]^d\,.
		\end{equation}
		Fix $d\ge1$. We can write
		\begin{align*}
			\EE\left[\int_0^t \sum_{1\le j<k\le n}\gamma(G_s^j-G_s^k+y_s^{jk})ds \right]^d
			&=\EE\int_{[0,t]^d}\prod_{q=1}^d\sum_{1\le j<k\le n}\gamma(G_{s_q}^j-G_{s_q}^k+y_{s_q}^{jk})ds_q
			\\&=\EE\sum_{\substack{1\le j_m<k_m\le n\\1\le m\le d}}\int_{[0,t]^d} \prod_{q=1}^d \gamma(G_{s_q}^{j_q}-G_{s_q}^{k_q}+y_{s_q}^{j_qk_q})ds_q\,.
		\end{align*}
		Using  formula \eqref{eqn:gamspec}, the right-hand side in the above equation  is the same as
		\begin{align*}
			\EE\sum_{\substack{1\le j_m<k_m\le n\\1\le m\le d}}\int_{[0,t]^d}\int_{(\RR^\ell)^d} e^{i \sum_{q=1}^d\xi^q\cdot(G_{s_q}^{j_q}-G_{s_q}^{k_q})}e^{i\sum_{q=1}^d \xi^q\cdot y_{s_q}^{j_qk_q}} {\mu}(d \xi) ds\,,
		\end{align*}
		where  $\mu(d\xi)=\mu(d \xi^1)\cdots \mu(d \xi^d)$ and $ds= ds_1 \cdots ds_d$. 
		It follows that 
		\begin{multline*}
			\EE\left[\int_0^t \sum_{1\le j<k\le n} a \gamma(G_s^j-G_s^k+y_s^{jk})ds \right]^d
			\le  |a|^d \left|\EE\left[\int_0^t \sum_{1\le j<k\le n}\gamma(G_s^j-G_s^k+y_s^{jk})ds \right]^d\right|
			\\=|a|^d\left|\sum_{\substack{1\le j_m<k_m\le n\\1\le m\le d}}\int_{[0,t]^d}\int_{(\RR^\ell)^d}\EE e^{i \sum_{q=1}^d\xi^q\cdot (G_{s_q}^{j_q}-G_{s_q}^{k_q})}e^{i\sum_{q=1}^d \xi^q\cdot y_{s_q}^{j_qk_q}}\mu(d\xi) ds\right|.
		\end{multline*}
		Applying the  triangle inequality, noting that $\EE e^{i \sum_{q=1}^d\xi^q\cdot(G_{s_q}^{j_q}-G_{s_q}^{k_q})}$ is strictly positive and $|e^{i\sum_{q=1}^d \xi^q\cdot y_{s_q}^{j_qk_q}}|=1$, the above quantity  is at most	
		\begin{equation*}
			|a|^d \EE\sum_{\substack{1\le j_m<k_m\le n\\1\le m\le d}}\int_{[0,t]^d}\int_{(\RR^\ell)^d} e^{i \sum_{q=1}^d\xi^q\cdot(G_{s_q}^{j_q}-G_{s_q}^{k_q})}\mu(d\xi) ds\,,
		\end{equation*}
		which is the same as $\EE\left[\int_0^t \sum_{1\le j<k\le n}|a|\gamma(G_s^j-G_s^k)ds \right]^d$,  by our argument at the beginning of the proof. Hence, we have shown \eqref{est:gaman} and the result follows.
	\end{proof}
	
	The next proposition is the key ingredient in the proof of the Feynman-Kac formula for the moments using Brownian  bridges. We recall that $\gamma_\ep$ is defined in  (\ref{eqn:ge}).
 
	\begin{proposition}\label{prop:expbridge}
		Suppose that the spectral measure satisfies  hypotheses    \ref{H1} or \ref{H2}. Let $\kappa$ be a real number. Then for each $\epsilon>0$, the function
		\begin{equation*}
			F_ \epsilon(x^1,\dots,x^n)=\EE\exp\left\{\kappa\int_0^t \sum_{1\le j<k\le n} \gamma_ \epsilon(B_{0,t}^j(s)- B^k_{0,t}(s)+x^j-x^k)ds\right\}
		\end{equation*}
		is well-defined and continuous. Moreover, as $\epsilon\downarrow0$, $F_ \epsilon$ converges uniformly. We denote the limiting function as
		\begin{equation*}
			\EE\exp\left\{\kappa\int_0^t \sum_{1\le j<k\le n} \gamma( B_{0,t}^j(s)- B^k_{0,t}(s)+x^j-x^k)ds\right\}\,.
		\end{equation*}
	\end{proposition} 
	
	\noindent
	{\bf Remark.} Actually, for each $1 \le j<k \le n$, the integral
	\[
	\int_0^t  \gamma( B_{0,t}^j(s)- B^k_{0,t}(s)+x^j-x^k)ds
	\]
	converges in $L^p(\Omega)$ as $\epsilon$ tends to zero, for each $p\ge 1$, and we can also denote the limit as
	\[
	\int_0^t\int_{\RR^\ell}  e^{i \xi\cdot(B_{0,t}^j(s)- B^k_{0,t}(s)+x^j-x^k)}\mu(d \xi)ds.
	\]

	\begin{proof} 
		 We claim that  for every $\kappa\in\RR$
		\begin{equation}  \label{k9}
			\sup_{\epsilon>0} \EE\exp\left\{\kappa\int_0^t \sum_{1\le j<k\le n} \gamma_ \epsilon(B_{0,t}^j(s)- B^k_{0,t}(s))ds\right\}<\infty\,.
		\end{equation}
		By H\"older inequality, it suffices to show the previous inequality for $n=2$. This is obtained by noting that $B^1_{0,t}+B^2_{0,t}\stackrel{\rm{law}}{=}\sqrt 2 B_{0,t}$ and the finiteness comes from Step 2 and 3 of Lemma \ref{Lemma:Conv chaos}.

		We now show that $F_ \epsilon$ converges uniformly as $\epsilon\downarrow0$. 
	Applying Lemma \ref{lem:gamma} and the  estimate (\ref{k9}), we see that for all  $\kappa\in \RR$
	\[
	\sup_{\epsilon>0} \sup_{x^1,\dots,x^n\in \RR^\ell} 	\EE  	 \exp\lt(\kappa\int_0^t \sum_{1\le j<k\le n} \gamma_ \epsilon(B_{0,t}^j(s)- B^k_{0,t}(s)+x^j-x^k)ds\rt)<\infty.
	\]
	As a consequence, applying the elementary inequality $e^a-e^b\le\frac12(e^a+e^b)(a-b)$, Cauchy-Schwarz inequality and the previous estimate, we obtain
	\begin{align*}
		|F_{\epsilon'}(x)-F_{\epsilon}(x)|\le C\lt[\EE\left(\sum_{1\le j<k\le n}\int_0^t(\gamma_{\epsilon}-\gamma_{\epsilon'})(B_{0,t}^j(s)-B^k_{0,t}(s)+x^j-x^k) ds\right)^2\rt]^{\frac12},
	\end{align*}
	for all $\epsilon'>\epsilon>0$.
  Together with Minkowski inequality, we see that $\sup_{x\in\RR^\ell}|F_{\epsilon'}(x)-F_{\epsilon}(x)|$ is at most a constant multiple of
	\begin{align*}
		&\sup_{x\in\RR^\ell}\sum_{1\le j<k\le n}\lt[\EE\left(\int_0^t(\gamma_{\epsilon}-\gamma_{\epsilon'})(B_{0,t}^j(s)-B_{0,t}^k(s)+x^j-x^k)ds \right)^2\rt]^{\frac12}
		% \\&=\sqrt{\frac{n(n-1)}2} \sup_{x\in\RR^\ell}\lt[\EE\left(\int_0^t(\gamma_{\epsilon}-\gamma_{\epsilon'})(B_{0,t}^1(s)-B_{0,t}^2(s)+x^j-x^k)ds \right)^2\rt]^{\frac12}
		\\&= \frac{n(n-1)}2  \lt[\EE\left(\int_0^t(\gamma_{\epsilon}-\gamma_{\epsilon'})(B_{0,t}^1(s)-B_{0,t}^2(s)) ds\right)^2\rt]^{\frac12}\,,
	\end{align*}
	where the last line {follows obviously from the proof of Lemma \ref{lem:gamma}}.   For every $\epsilon'>\epsilon>0$, as in the computation of (\ref{eqn:d.int.form}), we have
	\begin{multline}
	\EE\left(\int_0^t(\gamma_{\epsilon}-\gamma_{\epsilon'})(B_{0,t}^1(s)-B_{0,t}^2(s)) ds\right)^2     \\ 
	=2\int_{[0,t]^2_<} \int_{(\RR^\ell)^2} \exp\Bigg\{ \sum_{j=1}^2 - |\xi^1 + \cdots + \xi^j|^2 (s_{j+1}-s_j)  
	+\frac{1}{t} \left|\sum_{j=1}^2 (\xi^1 + \cdots+ \xi^j)(s_{j+1}-s_j) \right|^2 \Bigg\}
	\\\times\prod_{j=1}^2 \left(e^{-\epsilon|\xi^j|^2} - e^{-\epsilon'|\xi^j|^2}\right)\mu(d\xi^1)\mu(d\xi^2) ds_1ds_2\,. \label{eqn:2.int.form}
	\end{multline} 
 Then from the proof of Lemma  \ref{Lemma:Conv chaos} and the dominated convergence theorem, we see that as $\epsilon,\epsilon'\downarrow0$, the right-hand side  of (\ref{eqn:2.int.form}) converges to 0. It follows that as $\epsilon\downarrow0$, $F_ \epsilon$ converges uniformly to a continuous function. 
	\end{proof}

{Notice that the proof of  Proposition \ref{prop:expbridge} uses only the uniform exponential integrability and the inequality  $|e^x-e^y|\leq (e^x+e^y)|x-y|$. As a consequence, the result still holds if we replace Brownian bridges by Brownian motions, as it is shown in the next proposition, which has its own interest.}  

	\begin{proposition}\label{prop:expbrownian}
Suppose that the spectral measure satisfies  hypotheses \ref{H1} or \ref{H2}. Let $\kappa$ be a real number and $\{B^j(s),s\ge0\}$, $j=1,\dots,n$, be independent Brownian motions in $\RR^\ell$. Then as $\epsilon\downarrow0$, the random variables
		\begin{equation*}
			G_{\epsilon}(x^1,\dots,x^n):=\exp\left\{\kappa\int_0^t \sum_{1\le j<k\le n} \gamma_ \epsilon(B^j(s)- B^k(s)+x^j-x^k)ds\right\}
		\end{equation*}
		converge in $L^p(\Omega)$ for every $p\ge1$ to a random variable which we denote by
		\begin{equation*}
			\exp\left\{\kappa\int_0^t \sum_{1\le j<k\le n} \gamma( B^j(s)- B^k(s)+x^j-x^k)ds\right\}\,.
		\end{equation*}
	\end{proposition} 
{ 
\begin{proof}
Using similar arguments as Proposition \ref{prop:expbridge} first claim that for every $\kappa \in \RR$, 
\begin{equation}
\sup_{\epsilon > 0} \EE \exp \left \{ \kappa\int_0^t \sum_{1\le j<k\le n} \gamma_ \epsilon(B^j(s)- B^k(s)+x^j-x^k)ds \right\} < \infty\,. 
\end{equation}
By H\"older inequality, it suffices to show the previous inequality for $n=2$. For every    $d \in\NN$, we have 
	\begin{eqnarray} \notag
&&	\EE\left[   \int_0^t \gamma_ \epsilon(B^1(s)- B^2(s)+x^1-x^2)ds \right]^d \\ \notag 
  &=& \EE \left [ \int_0^t \int_{\RR^\ell} e^{i \xi \cdot(B^1(s)- B^2(s)+x^1-x^2)} \mu_{\epsilon}(d\xi) ds\right]^d\\ \notag
	&\leq&d!\int_{[0,t]^d_<} \int_{(\RR^\ell)^d} \exp\Bigg\{ \sum_{j=1}^d - |\xi^1 + \cdots + \xi^j|^2 (s_{j+1}-s_j) 
	 \Bigg\}\mu_\epsilon(d\xi) ds\\ \notag
	&\leq&d!\int_{[0,t]^d_<} \int_{(\RR^\ell)^d} \exp\Bigg\{ \sum_{j=1}^d - |\xi^1 + \cdots + \xi^j|^2 (s_{j+1}-s_j) 
	 \Bigg\}\mu(d\xi) ds\,,
	\label{eqn:d.int.form BM}
	\end{eqnarray}
using Taylor expansion and following the proof of Lemma \ref{Lemma:Conv chaos} we prove the claim. Then we follow the same lines of the proof of Proposition \ref{prop:expbridge} to get, for any $\epsilon'> \epsilon>0$, 
\begin{eqnarray*}
&&\left|G_{\epsilon}(x^1,\dots,x^n) - G_{\epsilon'}(x^1,\dots,x^n) \right|\\
&\leq&C \left[ \EE \left( \int_0^t (\gamma_\epsilon - \gamma_{\epsilon'}) (B^1(s)-B^2(s)) ds \right)^2  \right]^{1/2}\\
&\leq& C \bigg[ \int_{[0,t]^2_<} \int_{(\RR^\ell)^2} \exp\Bigg\{ \sum_{j=1}^2 - |\xi^1 + \cdots + \xi^j|^2 (s_{j+1}-s_j) \Bigg\} \\
&& \times \prod_{j=1}^2 \left(e^{-\epsilon|\xi^j|^2} - e^{-\epsilon'|\xi^j|^2}\right)\mu(d\xi^1)\mu(d\xi^2) ds_1ds_2\bigg]^{1/2}\,,
\end{eqnarray*}
then it follows from the proof of Lemma \ref{Lemma:Conv chaos} and dominated convergence theorem that as $\epsilon,\epsilon' \downarrow 0$, the above expression converges to $0$, this completes the proof.
\end{proof}	  }
	
	In the sequel we will make use of the notations
		\begin{equation}  \label{k12}
	\Gamma(x)=\sum_{1\le j<k\le n}\gamma(x^j-x^k) \quad  {\rm and} \quad 	\Gamma_ \epsilon(y)=\sum_{1\le j<k\le n}\gamma_ \epsilon(y^j-y^k).
		\end{equation}

	As an application, we have the following Feynman-Kac formula based on Brownian bridges.
	\begin{proposition}    \label{prop1}
		Assume conditions    \ref{H1} or \ref{H2}.  Suppose that the initial condition satisfies condition (\ref{eq1}). Suppose that $\{B_{0,t}^j(s),s\in[0,t]\}$, $j=1,\dots, n$ are independent Brownian bridges.
		Then for every $x^1,\dots,x^n\in\RR^\ell$,
		\begin{eqnarray}\label{eqn:FKbridge}
			\EE\left[\prod_{j=1}^n u(t,x^j)\right]
			&=&\int_{(\RR^\ell)^n}   \notag 
			\EE\exp\left\{\int_0^t\sum_{1\le j<k\le n} \gamma\left(B_{0,t}^j(s)-B_{0,t}^k(s)+x^j-x^k+\frac st(y^j-y^k)\right)ds\right\}  \\
			&&  \times \prod_{j=1}^n [u_0(x^j+y^j)p_t(y^j)]d y^1\cdots d y^n\,.
		\end{eqnarray}
	\end{proposition}
	\begin{proof} 
	For any $\ep>0$ we denote by $u_\ep(t,x)$ the solution to the stochastic heat equation
	\[
	\frac {\partial u_\ep }{\partial t }=\frac12  \Delta u_{\ep} +u_\ep \dot{W_\ep}\,,\quad u(0,\cdot)=u_0(\cdot)\,,
	\]
	where $W_\ep$ is a white noise in time and it has the spectral spatial measure $e^{-\ep |\xi|^2} \mu(\xi)$. From the results  of Conus \cite{Conus} we have the following Feynman-Kac formula for the moments of $u_\ep$ 
	 \begin{multline}\label{eqn:FKn}
		\EE\left[ \prod_{j=1}^n u_\ep(t,x^j) \right]
		\\=\EE \left( \prod_{j=1}^n u_0(B^j(t)+x^j)\exp\left\{\sum_{1\le j<k\le n}\int_0^t \gamma_\ep(B^j(s)-B^k(s)+(x^j-x^k) )ds \right\}\right),
	\end{multline}
	where $B^j$ are independent $\ell$-dimensional standard Brownian motions.  We remark that in \cite{Conus} it is required that $\gamma$ is a non-negative function, which is not necessarily true for $\gamma_\ep$. However,   $\gamma_\ep$ is bounded, and, in this case, it is not difficult to show that  (\ref{eqn:FKn})  still holds.
	
	For each $j=1,\dots,n$ and every fixed $t>0$, the Brownian motion $B^j$ admits the following decomposition
		\begin{equation}
			B^j(s)=B_{0,t}^j(s)+\frac st B^j(t),
		\end{equation}
		where $\{B_{0,t}^j(s),s\in[0,t]\}$, $j=1,\dots,n$ are Brownian bridges on $\mathbb{R}^\ell$ independent from $\{B^j(t), 1\le  j \le  n\}$ and from each other. Thus, identity \eqref{eqn:FKn} can be written as
		\begin{equation}  \label{eq5}
		\EE\left[ \prod_{j=1}^n u_\ep(t,x^j) \right]=\int_{(\RR^\ell)^n}
			 \prod_{j=1}^n [u_0(x^j+y^j)p_t(y^j)]\EE\exp\left\{\int_0^t\Gamma_ \epsilon(B_{0,t}(s)+x+\frac st y)ds\right\}dy\,,
			 \end{equation}
			 where $\Gamma_\ep$ is defined in (\ref{k12}).

		From Proposition \ref{prop:expbridge} and the  dominated convergence theorem, the right-hand side of 
		(\ref{eq5}) converges to the right-hand side of (\ref{eqn:FKbridge}).  From the Wiener chaos expansion and the computations in  the proof of Theorem \ref{thm1}, it follows easily that $u_\ep(t,x)$ converges in $L^2(\Omega)$ to $u(t,x)$. On the other hand, from (\ref{eq5}) it follows that the moments of all orders of $u_\ep(t,x)$ are uniformly bounded in $\ep$. As a consequence, the left-hand side of 
		(\ref{eq5}) converges to the left-hand side of (\ref{eqn:FKbridge}).  This completes the proof. 
	\end{proof}
 	In fact, for regular function $V$, Feynman-Kac formulas based on Brownian bridges for the solution of $\frac {\partial u}{\partial t}- \Delta u=uV$ are not new and we refer the  readers to \cite{JFV} and the references therein for further details and other applications. 
 
	\begin{remark}
	If the initial condition $u_0$ is nonnegative, one can show that $u(t,x) \ge 0$ a.s., for all $t\ge 0$ and $x\in \RR^\ell$. 
	This follows from the fact that $u_\ep(t,x)$ is nonnegative for any $\ep$, where $u_\ep$ is the random field introduced in the proof of   Proposition \ref{prop1}.
	\end{remark}

\section{{Large deviation and approximation of covariance function}}
\label{sec:largedev}

{ In this section we give some key results that are need in proving Theorem \ref{thm:largetimemoment}. We first give some results from large deviation theory, which can be applied to the case when our covariance is a continuous and bounded function, see Lemma \ref{lem:Vupper}. Then for the general covariance, we will approximate it using the regularized covariance function \eqref{eqn:ge}. }

 	\begin{lemma}\label{lem:Vupper}
		Let $\{B_{0,t}(s),s\in[0,t]\}$ be a Brownian bridge in $\RR^n$. Let $F: \RR^n\to\RR$ be a bounded continuous function.
		%  We assume that $F$ satisfies
		% \begin{equation}\label{con:gradGam}
		% 	\sup_{x\in\RR^n}|\nabla F(x)\cdot x|<\infty\,. 	
		% \end{equation}
		Let $o(1)$ be a quantity such that $o(1)\to0$ as $t\to\infty$. Then, for every fixed $x_0\in\RR^n$,
		\begin{align}
			&\lim_{t\to\infty}\frac1t\log\sup_{|y|\le o(1)t} \EE \exp\left\{\int_0^t F\left(B_{0,t}(s)+x_0+\frac st y\right)ds \right\}
			\nonumber\\&=\lim_{t\to\infty}\frac1t\log\inf_{|y|\le o(1)t} \EE \exp\left\{\int_0^t F\left(B_{0,t}(s)+x_0+\frac st y\right)ds \right\}
			\nonumber\\&={\cee(F)}\,,\label{est:VbyE}
		\end{align}
			where 
			\begin{equation*}
			\cee(F )=\sup_{g\in\cff_n}\left\{\int_{\RR^n} F(x) g^2(x)dx-\frac12\int_{\RR^n}|\nabla g(x)|^2dx \right\}.
		\end{equation*}
	\end{lemma}
	\begin{proof}
		We note that $\cee(F)=\cee(F(\cdot+x_0))$. It suffices to show  {that}
		\begin{equation}\label{Fup}
			\limsup_{t\to\infty}\frac1t\log\sup_{|y|\le o(1)t} \EE \exp\left\{\int_0^t F\left(B_{0,t}(s)+x_0+\frac st y\right)ds \right\}\le \cee(F(\cdot+x_0))
		\end{equation}
		and
		\begin{equation}\label{Flow}
			\liminf_{t\to\infty}\frac1t\log\inf_{|y|\le o(1)t} \EE \exp\left\{\int_0^t F\left(B_{0,t}(s)+x_0+\frac st y\right)ds \right\}\ge \cee(F(\cdot+x_0))\,.
		\end{equation}

		\textit{Upper bound:} Fix $\lambda\in(0,1)$. From \eqref{id:density}, we see that
		\begin{align*}
			\EE\exp\lt\{\int_0^{\lambda t}F(B_{0,t}(s)+x_0+\frac st y)ds \rt\}
			\le (1- \lambda)^{-\frac\ell2}\EE\exp\lt\{\int_0^{\lambda t}F(B(s)+x_0)ds+{\frac{|y|^2}{2t} }\rt\},
		\end{align*}
		%{\color{red}(here $B_{0,t}(s)+x_0+\frac st y$ is bridge from $x_0$ to $x_0+y$, so I change the density back to the original. Same for the lower bound below. ) }
		where $B$ is a Brownian motion. It is  {obvious} that
		\begin{equation*}
			\EE\exp\lt\{\int_{\lambda t}^tF(B_{0,t}(s)+x_0+\frac st y)ds \rt\}\le e^{\|F\|_\infty (1- \lambda)t}.
		\end{equation*}
		Together with \cite{Chenbook}*{Theorem 4.1.6}, we have
		\begin{align*}
			\limsup_{t\to\infty}\frac1t\log&\sup_{|y|\le o(1)t}\EE\exp\lt\{\int_0^{t}F(B_{0,t}(s)+x_0+\frac st y)ds \rt\}
			\\&\le (1- \lambda)\|F\|_\infty+	\limsup_{t\to\infty}\frac1t\log\EE\exp\lt\{\int_0^{\lambda t}F(B(s)+x_0)ds \rt\}
			\\&=(1- \lambda)\|F\|_\infty+\lambda \cee(F(\cdot+x_0)).
		\end{align*}
		By sending $\lambda\to1^-$, we obtain \eqref{Fup}.

		\textit{Lower bound:} Fix $\lambda\in(0,1)$ and $R>0$. Let $\{B(s);s\ge0\}$ be a Brownian motion and $A$ be the event  {$\{ \sup_{0\le s\le \lambda t}|B(s)|\le R\}$.}  Again from \eqref{id:density} we see that
		\begin{multline*}
			\EE\exp\lt\{\int_0^{\lambda t}F(B_{0,t}(s)+x_0+\frac st y)ds \rt\}
			\\\ge (1- \lambda)^{-\frac\ell2}\EE \lt[\mathbf{1}_{A}\exp\lt\{\int_0^{\lambda t}F(B(s)+x_0)ds+{\frac{|y|^2}{2t}-\frac{|y-B(\lambda t)|^2}{2t(1- \lambda)}}\rt\}\rt].
		\end{multline*}

		On the event $A$, we have for every $|y|\le o(1)t$, 
		\begin{equation*}
			{\frac{|y|^2}{2t}-\frac{|y-B(\lambda t)|^2}{2t(1- \lambda)}} \ge -o(1)C(R,\lambda,|x_0|) {t}
		\end{equation*}
		for some deterministic positive constant $C(R,\lambda,|x_0|)$ depend only on $R$, $\lambda$ and $|x_0|$. We also note that
		\begin{equation*}
			 \EE\exp\lt\{{\int_{\lambda t}^{t}}F(B_{0,t}(s)+x_0+\frac st y)ds \rt\}\ge e^{-\|F\|_\infty(1- \lambda) t}.
		\end{equation*}
		Upon combining these estimates together with the argument of \cite{Chetal15}*{Proposition 3.1}, it follows that
		\begin{align*}
			\liminf_{t\to\infty}\frac1t\log&\inf_{|y|\le o(1)t} \EE\exp\lt\{\int_0^{ t}F(B_{0,t}(s)+x_0+\frac st y)ds \rt\}
			\\&\ge-\|F\|_\infty(1- \lambda)+ \liminf_{t\to\infty}\frac1t\log\EE \lt[\mathbf{1}_A\exp\lt\{\int_0^{\lambda t}F(B(s)+x_0)ds \rt\}\rt]
			\\&\ge-\|F\|_\infty(1- \lambda)+ \lambda\sup\lt\{\int_{D} F(x+x_0)g^2(x) dx-\frac12\int_{D}|\nabla g(x)|^2dx \rt\}
		\end{align*}
		where $D=\{x\in\RR^\ell:|x|\le R\}$ and the supremum is taken over all smooth functions $g$ on $D$ such that $\|g\|_{L^2(D)}=1$ and $g\big|_{\partial D}=0$. By sending $\lambda\to 1^-$ and $R\to\infty$, we arrive at \eqref{Flow}.
	\end{proof}

	We recall that, for any $\ep>0$, $\gamma_\ep$ is the bounded covariance function defined in \eqref{eqn:ge}.
	\begin{lemma}\label{lem:ggep} Suppose that the spectral measure $\mu$ satisfies  hypotheses  \ref{H1}  or \ref{H2}. Then for every $a \in \RR$, $x_0=(x_0^1,\dots,x_0^n)\in(\RR^\ell)^n$,
		\begin{multline*}
			\lim_{\epsilon\downarrow0}\limsup_{t\to\infty}\frac1t\log\sup_{y^j\in\RR^\ell\,;j=1,\dots,n}
			\\\EE\exp\left\{a \int_0^t  \sum_{1\le j<k \le n}(\gamma- \gamma_ \epsilon)(B_{0,t}^j(s) -B_{0,t}^k(s)+(x^j_0-x^k_0) +\frac st(y^j-y^k))ds\right\}\le0.
		\end{multline*}
		% \begin{align*}
		% 	\EE\exp\left\{a \int_0^t  \sum_{1\le j<k \le n}(\gamma- \gamma_ \epsilon)(\wtb_s^j -\wtb^k_s+y^j({\frac st})- y^k(\frac st))ds\right\}\le   c_1 e^{c_2 \epsilon ^\alpha t}.
		% \end{align*}
	\end{lemma}
	\begin{proof}
	Thanks to Lemma \ref{lem:gamma},	we can assume that $a>0$ and $x=y=0$. In addition, by applying H\"older inequality, we can assume that $n=2$. By subadditivity (Lemma \ref{lem:Sub add BB}), we see that
		\begin{align*}
			\EE\exp\left\{a \int_0^t  (\gamma- \gamma_ \epsilon)(B_{0,t}^1(s)-B_{0,t}^2(s))ds\right\}
			\le C\lt(\EE\exp\left\{a \int_0^1  (\gamma- \gamma_ \epsilon)(B_{0,1}^1(s)-B_{0,1}^2(s))ds\right\}\rt)^{\lfloor t\rfloor}\,,
		\end{align*}
		where $\lfloor t\rfloor$ is the integer part of $t$ and
		\begin{align*}
			C=\sup_{\epsilon>0} \sup_{0\le s<1}\EE\exp\left\{a \int_0^s  (\gamma- \gamma_ \epsilon)(B_{0,s}^1(r)-B_{0,s}^2(r))dr\right\}\,.
		\end{align*}
		We observe that by similar computations to \eqref{eqn:d.int.form},
		\begin{equation*}
			C\le \sup_{0\le s<1}\EE\exp\left\{a \int_0^s  \gamma(B_{0,s}^1(r)-B_{0,s}^2(r))dr\right\}\,.
		\end{equation*}
		The term on the right is finite by Lemma \ref{Lemma:Conv chaos}. 
		To complete the proof, it remains to show that
		\begin{align}\label{tmp:aggep}
			\lim_{\epsilon\downarrow0} \EE\exp\left\{a \int_0^1  (\gamma- \gamma_ \epsilon)(B_{0,1}^1(s)-B_{0,1}^2(s))ds\right\}=0\,.
		\end{align}
		Indeed, by Taylor expansion and Tonelli's theorem
		\begin{align*}
			\EE\exp\left\{a \int_0^1  (\gamma- \gamma_ \epsilon)(B_{0,1}^1(s)-B_{0,1}^2(s))ds\right\}
			=\sum_{d=0}^\infty \frac1{d!}\EE\left(a\int_0^1  (\gamma- \gamma_ \epsilon)(B_{0,1}^1(s)-B_{0,1}^2(s))ds\right)^d.
		\end{align*}
		From the proof of Proposition \ref{prop:expbridge}, it follows that that the above series converges uniformly in $\epsilon$,  because the spectral measure of $\gamma-\gamma_\epsilon$ is uniformly  bounded by the spectral measure of $\gamma$, 		
		and that for each $d$, the expectation
		\begin{equation*}
			\EE\left(a\int_0^1  (\gamma- \gamma_ \epsilon)(B_{0,1}^1(s)-B_{0,1}^2(s))ds\right)^d
		\end{equation*}
		converges to 0 as $\epsilon$ tends to 0.  Hence, an application of dominated convergence theorem  yields \eqref{tmp:aggep}. 
	\end{proof}
 	\begin{lemma}\label{lem:e epsilon finite}
		Under the assumptions  \ref{H1} or \ref{H2}, 
		\begin{equation}\label{est:En}
		{0\le \inf_{\epsilon >0} \cee_n(\gamma_{\epsilon})\leq \sup_{\epsilon > 0}\cee_n(\gamma_{\epsilon})<\infty.}
		\end{equation}
	\end{lemma}      
	\begin{proof}
		Applying Jensen inequality, we see that
		\begin{align*}
			\EE\exp\lt\{\int_0^t\sum_{1\le j<k\le n}\gamma_{\epsilon}(B_s^j-B_s^k)ds\rt\}
			\ge \exp\lt\{\EE\int_0^t\sum_{1\le j<k\le n}\gamma_{\epsilon}(B_s^j-B_s^k)ds \rt\}\,.
		\end{align*}
		By the same argument as in Lemma \ref{lem:gamma}, the later term is at least 1. The first inequality in \eqref{est:En} follows. The second inequality is trivial. We begin showing the last inequality in \eqref{est:En}. We first observe that by the Markov property of Brownian motion and Lemma \ref{Lemma:Conv chaos}, the map 
		$$t\mapsto\log\EE\exp\lt\{\int_0^t\sum_{1\le j<k\le n}\gamma(B_s^j-B_s^k)ds\rt\} $$ is finite and subadditive. As a result, the limit
		\begin{equation*}
			\lim_{t\to\infty}\frac1t\log\EE\exp\lt\{\int_0^t\sum_{1\le j<k\le n}\gamma(B_s^j-B_s^k)ds\rt\}
		\end{equation*}
		exists and is finite. We denote the limiting value as $\Lambda$. By standard result on Feynman-Kac functional (see for instance \cite{Chenbook}*{Proposition 4.1.6}), for each $\epsilon>0$ we have
		\begin{align*}
			\cee_n(\gamma_ \epsilon)=\lim_{t\to\infty}\frac1t\log\EE\exp\lt\{\int_0^t\sum_{1\le j<k\le n}\gamma_ \epsilon(B_s^j-B_s^k)ds\rt\}\,.
		\end{align*}
		On the other hand, reasoning as in Lemma \ref{lem:gamma}, we have
		\begin{align*}
			\EE\exp\lt\{\int_0^t\sum_{1\le j<k\le n}\gamma_ \epsilon(B_s^j-B_s^k)ds\rt\}\le \EE\exp\lt\{\int_0^t\sum_{1\le j<k\le n}\gamma(B_s^j-B_s^k)ds\rt\}\,.
		\end{align*}
		As a result, $\cee_n(\gamma_ \epsilon)\le \Lambda$ for all $\epsilon>0$, which concludes the proof.
	\end{proof}
	We need the following convergence result.

	\begin{proposition}  \label{prop:Efinite}
	Suppose that the spectral measure $\mu$ satisfies  hypotheses \ref{H1} or \ref{H2}. Then,
	% \begin{enumerate}[(i)]
	% 	\item $\lim_{\ep \downarrow 0} \mathcal{E}_n(\gamma_\ep) = \mathcal{E}_n(\gamma)$
	% 	\item $0 \leq \mathcal{E}_n(\gamma) < \infty$
	% \end{enumerate}
	\begin{eqnarray}  \label{y24}
&(i)&	\lim_{\ep \downarrow 0} \mathcal{E}_n(\gamma_\ep) = \mathcal{E}_n(\gamma)\,,\\
&(ii)& 0 \leq \mathcal{E}_n(\gamma) < \infty\,.
	\end{eqnarray}	
	\end{proposition}    
	
	\begin{proof}
We claim that the following alternative representation of $\mathcal{E}_n(\gamma_\ep)$ holds
	 \begin{equation}  \label{y23}
	  \mathcal{E}_n(\gamma_\ep)= \sup_{h\in \mathcal{A}_{n\ell}}   \left( (2\pi)^{-\ell}\int_{\RR^{\ell}} \sum_{1\le j<k\le n}  e^{-\ep |\xi|^2}  (h*h)(e_{jk}(\xi)) \mu(d\xi) - \frac 12
	  \int_{\RR^{n\ell}} |\xi |^2 |h(\xi)|^2 d\xi \right),
	  \end{equation}
where we	 recall  that 
	 \begin{equation}
	 \mathcal{A}_{n\ell}= \left \{ h: \RR^{n\ell} \rightarrow \mathbb{C} \Big| \| h\|^2 _{L^2(\RR^{n\ell})} =1,  \int_{\RR^{n\ell}} |\xi |^2 |h(\xi)|^2 d\xi  <\infty \ \text{and} \ \overline{h(\xi)} = h(-\xi)\right\}\,,
	 \end{equation}
	  and for each $1\le j<k\le n$ and $\xi \in \RR^{\ell}$, $e_{jk}(\xi):=(\xi^{1}_{jk},\dots,\xi^{n}_{jk} )$ is the vector in $\RR^{n\ell}$ with all $\xi^{i}_{jk}$, $1\leq i \leq n$ equal to $0$ except $\xi_{jk}^k=\xi$ and $\xi_{jk}^j =-\xi$.   
	To show (\ref{y23}), for any  $1\le j<k\le n$  and  any $g\in \mathcal{F}_{n\ell}$, we can write
	  \begin{eqnarray*}
	  \int_{\RR^{n\ell}} \gamma_\ep(x^j -x^k) g^2(x)dx &=& \frac1{(2 \pi)^{\ell}}  \int_{\RR^{n\ell}} \int_{\RR^{\ell}}  e^{-\ep |\xi|^2} e^{i\xi  \cdot (x^j-x^k)}  g^2(x )\mu(d\xi) dx \\
	  &=&\frac1{(2 \pi)^{\ell}}  \int_{\RR^{\ell}}   e^{-\ep |\xi|^2}     \mathcal{F}g^2(e_{jk}(\xi)) \mu(d\xi).
	  \end{eqnarray*} 
	  Let us put $h=\frac1{{(2 \pi)^{ n \ell /2}}}\cff g$. Then, (\ref{y23}) follows from
	  $\| \nabla g\| _{L^2(\RR^{n\ell})} ^2=  \int_{\RR^{n\ell}} |\xi |^2 |h(\xi)|^2 d\xi $ and $ \mathcal{F}g^2=  h*h$. We note that the convolution $h * h$ is a continuous function, so the quantity $(h*h)(e_{jk}(\xi))$ in (\ref{y23}) is well-defined.

	  We can now proceed to the proof of the convergence (\ref{y24}). First, notice that, for any $h\in \mathcal{A}_{n\ell}$, $|h|$ is also in $\mathcal{A}_{n\ell}$, thus
	 \begin{equation*}
	\sup_{\epsilon>0} (2\pi)^{-\ell} \int_{\RR^{\ell}}\sum_{1 \leq j < k \leq n} e^{-\epsilon |\xi|^2} (|h|* |h|)(e_{jk}(\xi)) \mu(d\xi) \leq  \frac{1}{2} \int_{\RR^{n\ell}} |\xi|^2 |h(\xi)|^2 d\xi + \sup_{\epsilon>0}\mathcal{E}_n(\gamma_{\epsilon})
	 \end{equation*} 
	 which is finite by Lemma \ref{lem:e epsilon finite}.
	Then, it follows from monotone convergence theorem that the integration
\begin{equation*}
 (2\pi)^{-\ell} \int_{\RR^{\ell}}\sum_{1 \leq j < k \leq n}(|h|* |h|)(e_{jk}(\xi)) \mu(d\xi)
\end{equation*}	  
is well-defined and finite for all $h\in\mathcal{A}_{n\ell}$. Hence, we can apply the  dominated convergence theorem to obtain
	  \begin{eqnarray*}
	  \lim_{\ep \downarrow 0} \mathcal{E}_n(\gamma_\ep)  &\ge&  \lim_{\ep \downarrow 0}   (2\pi)^{-\ell}
	   \int_{\RR^{\ell}} \sum_{1\le j<k\le n}  e^{-\ep |\xi|^2}  (h*h)(e_{jk}(\xi)) \mu(d\xi) - \frac 12 
	  \int_{\RR^{n\ell}} |\xi |^2 |h(\xi)|^2 d\xi  \\
	 & =&   (2\pi)^{-\ell}   \int_{\RR^{\ell}} \sum_{1\le j<k\le n}    (h*h)(e_{jk}(\xi)) \mu(d\xi) - \frac 12 
	  \int_{\RR^{n\ell}} |\xi |^2 |h(\xi)|^2 d\xi.
\end{eqnarray*}
	  The converse inequality is obvious, because for any $h\in \mathcal{A}_{n\ell}$, $|h|$ also belongs to $\mathcal{A}_{n\ell}$ and
	  \[
	    \int_{\RR^{\ell}}   e^{-\ep |\xi|^2}  (h*h)(e_{jk}(\xi)) \mu(d\xi) \le  \int_{\RR^{\ell}}     (|h|*|h|)(e_{jk}(\xi)) \mu(d\xi).
 \]
 This completes the proof of $(i)$.
 
 Finally, $(ii)$ is a direct consequence of $(i)$ and Lemma \ref{lem:e epsilon finite}.  
	     \end{proof}

	\begin{proposition}\label{prop:expgam}  
		Suppose that the spectral measure satisfies hypotheses    \ref{H1} or \ref{H2} and $B^i_{0,t}(s)$, $1 \leq i \leq n$, are $n$ independent $\ell$-dimensional Brownian bridges. Denote $B_{0,t}(s):= (B^1_{0,t}(s), \dots,B^n_{0,t}(s))$ and $x=(x^1,\dots,x^n), y=(y^1,\dots,y^n)\in(\RR^\ell)^n$. For each $M>0$, let $A_M$ be the set defined in Theorem \ref{thm:largetimemoment}.
		 Then for every $x\in(\RR^\ell)^n$,
		\begin{align}
			&\lim_{t\to\infty}\frac1t\log\sup_{y\in(\RR^\ell)^n} \EE \exp\left\{\int_0^t \sum_{1\le j<k\le n}\gamma \left(B_{0,t}^j(s)-B_{0,t}^k(s)+(x^j-x^k)+ \frac st( y^j-y^k)\right)ds \right\}
			\nonumber\\&=\lim_{t\to\infty}\frac1t\log\inf_{y\in A_M} \EE \exp\left\{\int_0^t \sum_{1\le j<k\le n}\gamma \left(B_{0,t}^j(s)-B_{0,t}^k(s)+(x^j-x^k)+\frac st( y^j-y^k)\right)ds \right\}
			\nonumber\\&=\cee_n(\gamma)\,.\label{eqn:limexpbri}
		\end{align}
	\end{proposition}	
	\begin{proof} We recall that  $\Gamma$ and $\Gamma_ \epsilon$ are defined in  \eqref{k12}. By Lemma \ref{lem:gamma} and Proposition \ref{prop:expbridge}, it suffices to show
	\begin{equation}\label{Gup}
		\limsup_{t\to\infty}\frac1t\log\EE\exp\int_0^t \Gamma(B_{0,t}(s))ds\le \cee_n(\gamma)
	\end{equation}
	and
	\begin{equation}\label{Glow}
		\liminf_{t\to\infty}\frac1t\log\inf_{y\in A_M}\EE\exp\int_0^t \Gamma(B_{0,t}(s)+x+\frac st y)ds\ge \cee_n(\gamma)
	\end{equation}

	\textit{Upper bound:} For any $p,q>1$, $p^{-1}+q^{-1}=1$, applying H\"older inequality, we have
		\begin{multline*}
			\log\EE \exp\left\{\int_0^t \Gamma(B_{0,t}(s))ds\right\}
			\le \frac1p\log\EE\exp\left\{p\int_0^t \Gamma_ \epsilon(B_{0,t}(s))ds\right\}
			\\+\frac1q\log\EE\exp\left\{q\int_0^t (\Gamma- \Gamma_ \epsilon) (B_{0,t}(s))ds\right\}\,.
		\end{multline*}
		From Lemma \ref{lem:Vupper}, we see that 
		\begin{equation*}
			\lim_{\epsilon\downarrow0} \lim_{t\to\infty} \frac1t\log\EE\exp\left\{p\int_0^t \Gamma_ \epsilon(B_{0,t}(s))ds\right\}= \cee_n(p {\gamma_ \epsilon}).
		\end{equation*}
		Moreover, by Proposition \ref{prop:Efinite},
		\[
	{\lim_{\epsilon\downarrow0}\cee_n(p \gamma_ \epsilon)=\cee_n(p \gamma).}
		\]
		On the other hand, it follows from Lemma \ref{lem:ggep} that
		\begin{equation*}
			\lim_{\epsilon\downarrow0}\limsup_{t\to\infty} \frac1t\log\EE\exp\left\{q\int_0^t (\Gamma- \Gamma_ \epsilon) (B_{0,t}(s))ds\right\}\le0\,.
		\end{equation*}
		Altogether we obtain
		\begin{equation*}
			\limsup_{t\to\infty}\frac 1t\log\EE\exp\left\{\int_0^t \Gamma (B_{0,t}(s))ds\right\}\le \frac1p\cee_n(p \gamma)
		\end{equation*}
		for any $p>1$. {Since $\cee_n(p\gamma)$ is monotone in $p$ by its definition, we may send $p\downarrow 1$ to get \eqref{Gup}.} 

		\textit{Lower bound:} We follow a similar argument. From H\"older inequality, we have
		\begin{align*}
			&p\log\inf_{y\in A_M} \EE \exp\left\{\frac1p\int_0^t \Gamma_ \epsilon(B_{0,t}(s)+x+\frac st y)ds\right\}
			\\&\le \log\inf_{y\in A_M}\EE\exp\left\{\int_0^t \Gamma(B_{0,t}(s)+x+\frac st y)ds\right\}
			\\&\quad+\frac pq\log\sup_{y\in A_M}\EE\exp\left\{-\frac qp\int_0^t (\Gamma-\Gamma_ \epsilon) (B_{0,t}(s)+x+\frac st y)ds\right\}\,.
		\end{align*}
		Using  the previous arguments,  we obtain \eqref{Glow}.
		This completes the proof.
	\end{proof}
	With the same methodology, we have the following result whose proof is left to  the readers. 
	\begin{proposition}\label{prop:expBM}
		Suppose that the spectral density satisfies hypotheses   \ref{H1} or \ref{H2}.   Let $\{B^i_s,s\ge0\}$, $1\leq i \leq n$, be n independent Brownian motions in $\RR^{\ell}$. Then for every $y=(y^1,\dots,y^n)\in (\RR^{\ell})^n$,
		\begin{equation}
			\lim_{t\to\infty}\frac1t\log\EE \exp\left\{\int_0^t \sum_{1\le j<k\le n}\gamma \left( B_s^j- B_s^k+y^j-y^k\right)ds \right\}=\cee_n(\gamma)\,.
			\end{equation}
	\end{proposition}	
\section{Proof of Theorem \ref{thm:largetimemoment}}\label{sec:proof12}
We prove Theorem \ref{thm:largetimemoment} in this section. For the sake of conciseness, we assume that $\ell=1$. The case of higher dimension works out analogously. We recall the notation of $\Gamma$  and $\Gamma_{\epsilon}$  defined in \eqref{k12}.
	\begin{proof}[Proof of \eqref{est:upper}]
		For every fixed $\epsilon>0$ and $x\in\RR^n$, from Lemma \ref{lem:gamma}, we see that
		\begin{equation*}
			\EE\exp\left\{\int_0^t\Gamma_ \epsilon(B_{0,t}(s)+x)ds\right\}\le\EE\exp\left\{\int_0^t\Gamma_ \epsilon(B_{0,t}(s))ds\right\}\,.
		\end{equation*}
		By passing through the limit $\epsilon\downarrow0$, employing Proposition \ref{prop:expbridge}, we obtain
		\begin{equation*}
			\EE\exp\left\{\int_0^t\Gamma(B_{0,t}(s)+x)ds\right\}\le\EE\exp\left\{\int_0^t\Gamma(B_{0,t}(s))ds\right\}
		\end{equation*}
		for all $x\in\RR^n$. Hence, applying the previous inequality in \eqref{eqn:FKbridge} yields
		\begin{align}\label{est:G0}
		\EE\left [\prod_{j=1}^n u(t,x^j)\right]\le \prod_{j=1}^n (p_t*u_0)(x^j)
	\EE\exp\left\{\int_0^t\Gamma(B_{0,t}(s))ds\right\}\,.
		\end{align}
		It follows that
		\begin{equation}\label{eqn:nthbyV}
			\sup_{x\in\RR}\frac{\EE\left [\prod_{j=1}^n u(t,x^j)\right] }  {\prod_{j=1}^n (p_t*u_0)(x^j)} \le \EE\exp\left\{\int_0^t\Gamma(B_{0,t}(s))ds\right\}\,.
		\end{equation}
		Applying  Proposition \ref{prop:expgam} yields \eqref{est:upper}.
	\end{proof}
	
	The lower bound requires a little bit more work.
	\begin{proof}[Proof of \eqref{est:lower}] 
		We consider the map $\eta:\RR^n\to\RR^{n(n-1)/2}$, $\eta(x)=(x^j-x^k)_{1\le j<k\le n}$. 
		For every fixed $M'>0$, clearly $A_{M'}=\eta^{-1}([-M',M']^{n(n-1)/2})$. 
		Let $\delta$ be a positive number. Let $\{I_h\}_{h=1}^N$ be a finite disjoint partition of $[-M',M']^{n(n-1)/2}$ such that for each $1\le h\le N$, $I_h$ is nonempty  and $\mathrm{diag }(I_h)\le \delta$. 
		% We define $I'_{k}$ for each $1\le k\le N'$ analogously with $M$ replaced by $M'$. 
		For each $h=1,\dots,N$, we choose an arbitrary $x_h=(x_h^j)_{j=1}^n$ in $\eta^{-1}(I_h)$.   It is useful to note that for every  $z\in\RR^n$, the expression
		\begin{equation*}
			\EE \exp\left\{\int_0^t \Gamma(B_{0,t}(s)+z)ds\right\}
		\end{equation*}
		depends only on $\eta(z)$.

		For each $\epsilon>0$, $p,q>1$ such that $p^{-1}+q^{-1}=1$, applying H\"older inequality, then Lemma \ref{lem:ggep}, we see that
		\begin{eqnarray*}
			\EE  \exp \left\{ \int_0^t \Gamma(B_{0,t}(s)+x+\frac st y)ds \right\}
			&\ge & \left(\EE  \exp \left\{  \frac1p\int_0^t \Gamma_ \epsilon(B_{0,t}(s)+x+\frac st y)ds \right\} \right)^p\\
			&& \times
			\left(\EE  \exp \left\{ -\frac qp\int_0^t(\Gamma -\Gamma_ \epsilon)(B_{0,t}(s)+x+\frac st y)ds \right\} \right)^{-\frac pq }   \\
			&\ge& \left(\EE  \exp \left\{  \frac1p\int_0^t \Gamma_ \epsilon(B_{0,t}(s)+x+\frac st y)ds \right\} \right)^p (\Phi_ \epsilon(t))^{-\frac pq},
		\end{eqnarray*}
		where $\Phi_ \epsilon(t)= \EE  \exp \left\{ \frac qp\int_0^t(\Gamma -\Gamma_ \epsilon)(B_{0,t}(s))ds\right\}$.
		For each $h=1,\dots,N$, $x\in \eta^{-1}(I_h)$, we have
		\begin{eqnarray*}
			& & \EE   \exp \left\{  \frac1p\int_0^t \Gamma_ \epsilon(B_{0,t}(s)+x+\frac st y)ds \right\}
			\\  & & \quad = \EE  \Bigg( \exp \left\{  \frac1p\int_0^t \Gamma_ \epsilon(B_{0,t}(s)+x_h+\frac st y)ds  \right\} \\
			&& \qquad \times \exp \left\{ \frac1p\int_0^t [\Gamma_ \epsilon(B_{0,t}(s)+x+\frac st y)-\Gamma_ \epsilon(B_{0,t}(s)+x_h+\frac st y)]ds \right \} \Bigg)
			\\&& \quad \ge   e^{-\frac1{p}\|\nabla\Gamma_ \epsilon\|_{\infty}\delta t} \EE  \exp \left\{\frac1p\int_0^t \Gamma_ \epsilon(B_{0,t}(s)+x_h+\frac st y)ds \right\}\,.
		\end{eqnarray*}
		Thus, for each $x\in A_{M'}$ and $y\in A_{M}$,
		\begin{multline*}
			\EE   \left\{ \exp  \int_0^t \Gamma(B_{0,t}(s)+x+\frac st y)ds \right\}
			\\\ge (\Phi_ \epsilon(t))^{-\frac pq }  e^{ - \frac 12 \|\nabla\Gamma_ \epsilon\|_{\infty}\delta t}
			\left(\inf_{1\le h\le N}\inf_{y\in A_M} \EE  \exp \left\{ \frac1p\int_0^t \Gamma_ \epsilon(B_{0,t}(s)+x_h+\frac st y)ds \right\}\right)^p
			 \,.
		\end{multline*}
		Hence, it follows from \eqref{eqn:FKbridge} that for every $x\in A_{M'}$
		\begin{align*}
			\EE \left[\prod_{j=1}^n u(t,x^j)\right]
			&\ge \int_{A_M}  \prod _{j=1}^n u_0(y^j)p_t(y^j-x^j) \EE\left(\exp\int_0^t \Gamma(B_{0,t}(s)+x+\frac st (y-x))ds\right)d y
			\\&\ge  (\Phi_ \epsilon(t))^{-\frac pq }e^{- \frac 12 \|\nabla\Gamma_ \epsilon\|_{\infty}\delta t} \\
			&\quad \times
			\min_{h\in\{1,\dots, N\}} \left(\inf_{y\in A_M}\EE  \exp \left\{ \frac1p\int_0^t \Gamma_ \epsilon(B_{0,t}(s)+x_h+\frac st (y-x_h))ds\right\}\right)^p\\
			&\quad\times
			\left(\int_{A_M}  \prod _{j=1}^n u_0(y^j)p_t(y^j-x^j) dy \right)\,.
		\end{align*}
		The above estimate yields
		\begin{multline*}
			\liminf_{t\to\infty}\frac1t\log\inf_{x\in A_{M'}} \frac{\EE\left[\prod_{j=1}^n u(t,x^j)\right]}{ \int_{A_M}  \prod _{j=1}^n u_0(y^j)p_t(y^j-x^j) dy}
			\\\ge \liminf_{t\to\infty}\frac pt\log \min_{h\in\{1,\dots, N\}}\inf_{y\in A_M}\EE  \exp \left\{ \frac1p\int_0^t \Gamma_ \epsilon(B_{0,t}(s)+x_h+\frac st (y-x_h))ds \right\}
			\\ -\frac pq \limsup_{t\to\infty}\frac1t\log \Phi_ \epsilon(t) - \frac 12 \|\nabla\Gamma_ \epsilon\|_{\infty}\delta \,.
		\end{multline*}
		On the other hand, for every fixed $M,M'$ and $\delta$,  $N,N'$ are fixed finite numbers, thus  applying Lemma \ref{lem:Vupper} we obtain
		\begin{equation*}
			\lim_{t\to\infty}\frac1t\log\min_{h\in\{1,\dots, N\}}\inf_{y\in A_M} \EE \exp  \left\{\frac1p\int_0^t \Gamma(B_{0,t}(s)+x_h+\frac st y)ds \right\}=\cee_n(p^{-1}\Gamma_ \epsilon)\,.
		\end{equation*}
		Therefore, for every $\epsilon,\delta>0$,
		\begin{multline*}
			\liminf_{t\to\infty}\frac1t\log\inf_{x\in A_{M'}} \frac{\EE \left[\prod
			_{j=1}^nu(t,x^j)\right]}{\int_{A_M}  \prod _{j=1}^n u_0(y^j)p_t(y^j-x^j) dy}
			\\\ge p\cee_n(p^{-1} \Gamma_ \epsilon)  -\frac pq \limsup_{t\to\infty}\frac1t\log \Phi_ \epsilon(t)- \frac 12 \|\nabla \Gamma_ \epsilon\|_\infty \delta\,.
		\end{multline*}
		Thanks to Lemma \ref{lem:ggep}, we can send $\delta\downarrow0$, then $\epsilon\downarrow0$, and finally $p\downarrow1$ to obtain \eqref{est:lower}.
	\end{proof}
	
\section{Speed of propagation of high peaks - Proof of Theorem \ref{thm:speed} } % (fold)
	\label{sec:propagation}

\begin{proof}[Proof of Theorem \ref{thm:speed}]
	We first show the upper bound  \eqref{est:lup}.	We observe that
	\begin{equation*}
		\sup_{|x|\ge \alpha t}\EE u^n(t,x)\le \left(\sup_{y\in\RR^\ell}\frac{\EE u^n(t,y)}{(p_t*u_0)(y)^n}\right) \left(\sup_{|x|\ge \alpha t}(p_t*u_0)(x)^n\right)\,.
	\end{equation*}
	It follows from \eqref{est:upper} that
	\begin{align*}
		\limsup_{t \to \infty} \frac{1}{t} \log \sup_{|x|\geq \alpha t} \EE u^n(t,x)
		&\le \cee_n(\gamma)+n\limsup_{t\to \infty}\frac1t\log\sup_{|x|\ge \alpha t} \int_{\RR^\ell}\exp(-\frac{|y-x|^2}{2t})u_0(y)dy\,.
	\end{align*}
	Let $\beta$ be a positive number such that $L:=\int_{\RR^\ell} e^{\beta|y|}u_0(y)dy$ is finite. Then, using the inequality
	\begin{equation*}
		\exp\{-\frac{|y-x|^2}{2t}-\beta|y| \}\le \exp\{\frac{\beta^2 t}2-\beta|x| \},
	\end{equation*}
	we get
	\begin{align*}
		\sup_{|x|\ge \alpha t} \int_{\RR^\ell}\exp(-\frac{|y-x|^2}{2t})u_0(y)dy
		\le L\exp\{\frac{\beta^2 t}2-\beta \alpha t \}
	\end{align*}
	and, hence,
	\begin{align*}
		\limsup_{t\to\infty}\frac1t {\log }\sup_{|x|\ge \alpha t} \int_{\RR^\ell}\exp(-\frac{|y-x|^2}{2t})u_0(y)dy
		\le \frac{\beta^2 }2-\beta \alpha \,.
	\end{align*}
	It follows that
	\begin{align*}
		\limsup_{t \to \infty} \frac{1}{t} \log \sup_{|x|\geq \alpha t} \EE u^n(t,x)
		&\le \cee_n(\gamma)+\frac{n\beta^2 }2-n\beta \alpha  \,.
	\end{align*}
	The above inequality shows
	\begin{equation*}
		\lambda^*(n) \le\frac{\cee_n(\gamma)}{n \beta}+\frac \beta2\,.
	\end{equation*}
	Since $\beta$ can be chosen arbitrarily so that $\int_{\RR^\ell} e^{\beta|y|}u_0(y)dy<\infty$, we derive \eqref{est:lup}.

	Next, we show \eqref{est:llow}. Fix $M>0$. We start with the estimate
	\begin{equation}\label{tmp:M1}
		\sup_{|x|\ge \alpha t}\EE  u^n(t,x)\ge \left(\inf_{x\in\RR ^\ell }\frac{\EE  u^n(t,x)}{\int_{A_M}\prod _{j=1}^n u_0(y^j)p_t(y^j-x) dy}\right) {\left(\sup_{|x|\geq \alpha t}\int_{A_M}\prod _{j=1}^n u_0(y^j)p_t(y^j-x) dy\right)\,.}
	\end{equation}
	Since $u_0$ is nontrivial, we can choose a positive number $K$ sufficiently large so that the integral $\int_{[-K,K]^\ell}u_0(y)dy$ is  nonzero. We then choose $M$ such that $A_M$ contains $([-K,K]^\ell)^n$. From \eqref{tmp:M1}, applying \eqref{est:lower}, we obtain
	\begin{equation*}
		\liminf_{t\to\infty}\frac1t\log\sup_{|x|\ge \alpha t}\EE u^n(t,x)\ge \cee_n(\gamma)+n\liminf_{t\to\infty}\frac 1t\log\sup_{|x|\ge \alpha t}\int_{[-K,K]^\ell} u_0(y)p_t(y- x)dy \,.
	\end{equation*}
	For every $|x|\ge \alpha t$ and $y\in [-K,K]^\ell$, {by triangle inequality}, $\exp\{-\frac{|y-x|^2}{2t}\}$ is at least $\exp\{-\frac{(K+\alpha t)^2}{2t}\}$,   %{\color {blue}  (maybe we need to provide some explanation about  point (18) in the letter to the referee but not here)}
thus
	\begin{multline*} 
		\liminf_{t\to\infty}\frac 1t\log\sup_{|x|\ge \alpha t}\int_{[-K,K]^\ell} u_0(y)p_t(y- x)dy
		\\\ge \liminf_{t\to\infty}\frac 1t\log\left((2 \pi t)^{-\frac\ell2} e^{-\frac{(K+ \alpha t)^2}{2 t}} \int_{[-K,K]^\ell} u_0(y)dy\right)=-\frac{\alpha^2}2\,.
	\end{multline*}
	It follows that
	\begin{equation*}
		\liminf_{t\to\infty}\frac1t\log\sup_{|x|\ge \alpha t}\EE u^n(t,x)\ge \cee_n(\gamma)-\frac{n \alpha^2}2,
	\end{equation*}
	which in turn, implies \eqref{est:llow}. 

	Finally, we show the estimate \eqref{est:llowbeta}. We put $\bar \alpha=(\alpha ,0,\dots,0)\in\RR^\ell$.
	Applying \eqref{tmp:M1}, \eqref{est:lower} together with our current assumption on $u_0$, we obtain
	\begin{align*}
		\liminf_{t\to\infty}\frac{1}{t} \log\sup_{|x|\ge \alpha t}\EE u^n(t,x)\ge \cee_n(\gamma)+\liminf_{t\to\infty}\frac1t\log \left(\int_{A_M^+}\prod _{j=1}^n e^{-\beta|y^j|} p_t(y^j-\bar \alpha t ) dy \right)
	\end{align*}
	where we have set $A_M^+=\{y\in A_M:y^j\in [0,\infty)^\ell,\,\forall j=1,\dots,n\}$. Together with Lemma \ref{lem:ldev}, we have
	\begin{align*}
		\liminf_{t\to\infty}\frac{1}{t}\log\sup_{|x|\ge \alpha t}\EE   u^n(t,x)\ge \cee_n(\gamma)+
		\left\{
		\begin{array}{lll}
			\frac {n\beta^2}2- n\beta \alpha & \text{ if } & \alpha> \beta\\
			 -\frac {n\alpha^2}{2} & \text{ if } & \alpha\le \beta
			 \end{array}
		\right.\,.
	\end{align*}
	In the case $\beta^2< 2\cee_n(\gamma)/n$, for every $\alpha$ in $(\beta,\frac{\cee_n(\gamma)}{n \beta}+\frac \beta2)$ it follows
	\begin{equation*}
		\liminf_{t\to\infty}\frac{1}{t}\log\sup_{|x|\ge \alpha t}\EE u^n(t,x)>0\,.
	\end{equation*}
	This shows $\lambda_*(n)\ge\frac{\cee_n(\gamma)}{n \beta}+\frac \beta2$.
	In the case $\beta^2\ge 2\cee_n(\gamma)/n$, for every $\alpha< \sqrt{\frac{2 \cee_n(\gamma)}n} $ we can write
	\begin{equation*}
		\liminf_{t\to\infty} \frac{1}{t}\log\sup_{|x|\ge \alpha t}\EE  u^n(t,x)\ge \cee_n(\gamma)-\frac{n \alpha^2}2> 0\,.
	\end{equation*}
	This shows $\lambda_*(n)\ge \sqrt{\frac{2 \cee_n(\gamma)}n}$. Summarizing the two cases, we obtain \eqref{est:llowbeta}. This completes the proof.
\end{proof}		

We conclude the current section with a few observations. The proof of Theorem \ref{thm:speed} indeed yields more information. Let us consider two separate types of initial conditions.
\begin{enumerate}[leftmargin=0cm,itemindent=1cm,label=\bf{\arabic*.}]
	\item $u_0$ is nontrivial and compactly supported. Then, for every $\alpha>0$, we have
	\begin{equation}
		\lim_{t\to\infty}\frac1t\log\sup_{|x|\ge \alpha t}\EE u^n(t,x)=\cee_n(\gamma)-\frac{n \alpha^2}2.
	\end{equation}
	\item $u_0$ satisfies $c_1e^{-\beta|y|}\le u_0(y)\le c_2 e^{-\beta|y|}$ for some positive numbers $c_1,c_2$ and $\beta$. Then, for every $\alpha>0$, we have
	\begin{equation}
		\lim_{t\to\infty}\frac1t\log\sup_{|x|\ge \alpha t}\EE u^n(t,x)=\cee_n(\gamma)+
		\left\{
		\begin{array}{lll}
			\frac {n\beta^2}2- n\beta \alpha & \text{ if } & \alpha> \beta\\
			 -\frac {n\alpha^2}{2} & \text{ if } & \alpha\le \beta
			 \end{array}
		\right.\,.
	\end{equation}
\end{enumerate}

% So if $\alpha > \frac{\beta'}{2}+ \frac{\cee_n(\gamma)}{n \beta'}$, then the right hand side of the above inequality is less than $0$. Because of our constraint $\beta'< \beta$, taking the infimum on $\alpha$, we have
% \begin{eqnarray*}
% \inf_{\alpha}\leq 
% \begin{cases}
% \sqrt{\frac{2\cee _n(\gamma)}{n}}, &\text{if $\beta > \sqrt{\frac{2\cee_n(\gamma)}{n}}$}\,, \\
% \frac{\beta}{2}+ \frac{\cee_n(\gamma)}{n\beta}, &\text{if $\beta \leq \sqrt{\frac{2\cee_n(\gamma)}{n}}$} \,.
% \end{cases}
% \end{eqnarray*}
% section propagation (end)	
\section{Phase transition} % (fold)
\label{sec:phase}
	In the current section, we give the proofs of Proposition \ref{prop:phase}, and Theorems \ref{thm:front1} and \ref{thm:front2}. First, we need the following lemma. 
	\begin{lemma}\label{lem:epeq}
		For all integers $1< p\le q$, the following inequalities hold
		\begin{equation}\label{est:e1}
			\frac{\cee_p(\gamma)}{p}\le\frac{\cee_q(\gamma)}{q}\,,
		\end{equation}
 		\begin{equation}\label{est:e2}
			\frac{\cee_H((q-1) \gamma)}2\le \frac{\cee_q(\gamma)}q\le \frac{\cee_p({\frac{q-1}{p-1}}\gamma)}{p}\,,
		\end{equation}        
		where for every $\lambda>0$, $\cee_H(\lambda \gamma) $ is defined as
		\begin{equation*}
			\cee_H(\lambda\gamma)=\sup_{h\in\mathcal{A}_\ell}\left\{\lambda (2 \pi)^{ -\ell}\int_{\RR^\ell}|(h*h)(\xi)|^2 \mu(d \xi)-\int_{\RR^\ell}|h(\xi)|^2|\xi|^2 d \xi \right\}.
		\end{equation*}    
	\end{lemma}
	\begin{proof}
		The estimate \eqref{est:e1} is obtained by applying H\"older's inequality together with \eqref{tmp:nmmt}.

 	The first inequality in \eqref{est:e2} is obtained by restricting the supremum in \eqref{id:ehat} to  the set $\{(x^1,\dots,x^n)\mapsto h(x^1)\cdots h(x^n): h\in\mathcal{A}_\ell\}$.    
		We will use hypercontractivity to show the second inequality in \eqref{est:e2}. Let $\{P_\tau\}_{\tau\ge0}$ denote the Ornstein-Uhlenbeck semigroup in the Gaussian space associated with the noise $W$.  For a bounded measurable function $f$ on $\RR^{\HH}$,   we have the following Mehler's formula
		\begin{equation*}
			P_\tau f(W) =\EE'f(e^{-\tau}W+\sqrt{1-e^{-2\tau}}W')\,,
		\end{equation*}
		where  $W'$  an independent copy of $W$, and $\EE'$ denotes the expectation with respect to $W'$. 
 For each $\tau\ge0$, let $u_{\tau,\lambda}$ be the solution to  equation \eqref{eqn:SHE} driven by the space-time Gaussian field $\sqrt{\lambda} (e^{-\tau}W+\sqrt{1-e^{-2\tau}}W')$, with initial condition $u_0=1$. That is,
		\begin{align*}
			u_{\tau,\lambda}(t,x)&=1
			\\&\quad+\sqrt\lambda\int_0^t\int_{\RR ^\ell} p_{t-s}(x-y)u_{\tau,\lambda}(s,y) (e^{-\tau}W(ds,dy)+\sqrt{1-e^{-2\tau}}W'(ds,dy))
			\,.
		\end{align*}
		From Mehler's formula we see that $P_\tau u_\lambda=\EE_{W'}[ u_{\tau,\lambda}]$ satisfies the equation
		\begin{equation*}
			P_\tau u_\lambda(t,x)=1+\sqrt\lambda e^{-\tau}\int_0^t\int_{\RR^\ell} p_{t-s}(x-y)P_\tau u_\lambda(s,y)W(ds,dy)\,.
		\end{equation*}
		In other words, $P_ \tau u_ \lambda$ is another solution of \eqref{eqn:ulambda}. By uniqueness, we conclude that $P_\tau u_\lambda=u_{e^{-2\tau}\lambda}$. On the other hand, it is well-known that the Ornstein-Uhlenbeck semigroup verifies the following hypercontractivity inequality
		\begin{equation}\label{ineq.hypercontractivity}
			\|P_{\tau}f\|_{q(\tau)}\le\|f\|_{p}
		\end{equation}
		for all $1<p<\infty$ and  $\tau\ge0$, where  $q(\tau)=1+e^{2 \tau}(p-1)$. Hence, applying the hypercontractivity property to our situation yields
		\begin{equation}\label{tmp:utl}
			\|u_{e^{-2\tau}\lambda}(t,x)\|_{q(\tau)}\le\|u_{\lambda}(t,x)\|_p
		\end{equation}
		for all $t>0,x\in\RR^\ell$ and $\tau\ge 0$. We now choose $\tau$ so that $q(\tau)=q$. Passing through the limit $t\to\infty$ in \eqref{tmp:utl}, employing \eqref{tmp:nmmt} and the relation $e^{-2\tau}=\frac{p-1}{q-1}$, we obtain
		\begin{equation*}
		 		\frac1q\cee_q(\frac{p-1}{q-1}\lambda \gamma)\le \frac1p\cee_p(\lambda \gamma),
		\end{equation*}
		which is clearly equivalent to \eqref{est:e2}.
	\end{proof}
	\begin{remark}
		(i) In \cite{B}, the author uses hypercontractivity on each Wiener chaos to obtain an upper bound for $p$-th moment of $u$. Comparing with the use of hypercontractivity in proving \eqref{est:e2}, these two methods certainly share the same flavor. Here, we applied hypercontractivity directly on the solution instead of on each Wiener chaos. Our method avoids some unnecessary constants and is therefore more transparent.

		(ii) It is interesting to observe that \eqref{est:e2} implies that $n\mapsto \frac{\cee_n(\frac{\gamma}{n-1})}n$ is nonincreasing. Thus, the limit $\lim_{n\to\infty}\frac{\cee_n(\frac{\gamma}{n-1})}n$ always exists. In fact, it is shown in \cite{ChPh15} that under our current assumptions (\ref{H1} or \ref{H2})
	\begin{equation*}
		\lim_{n\to\infty}\frac{\cee_n(\frac{\gamma}{n-1})}n=\frac{\cee_H(\gamma)}2\,.
	\end{equation*}
	Suppose in addition that $\mu$ has a scaling property: there exists $\alpha\neq 2$ such that $\mu(d(c \xi))=c^\alpha \mu(d \xi)$ for all $c>0$. This implies that $\cee_2$ and $\cee_H$ have  the scaling property $\cee_2(\lambda \gamma)=\lambda^{\frac{2}{2- \alpha}}\cee_2(\gamma)$ and  $\cee_H(\lambda \gamma)=\lambda^{\frac{2}{2- \alpha}}\cee_H(\gamma)$. Therefore, estimate \eqref{est:e2} (with $p=2$) becomes
	\begin{equation}\label{est:Ehen}
		\frac12q(q-1)^{\frac2{2- \alpha}}\cee_H(\gamma) \le \cee_q(\gamma)\le \frac12q(q-1)^{\frac2{2- \alpha}}\cee_2(\gamma)\,.
	\end{equation}

	(iii) Finally, let us mention that in other situations, which are not considered in the current paper, the Lyapunov exponent of $\EE u^n(t,x)$ can be computed explicitly. In fact, Chen  have recently proved in \cite{Chen}  that
	\begin{align}
			\lim_{t\to\infty}\frac 1{t^{\frac{4- \alpha-2 \alpha_0}{2- \alpha}}}\log \EE  u^n(t,x)= n \left(\frac {n-1}2  \right) ^{\frac 2{2-\alpha}} \mathcal{E},
		\end{align}
		assuming that $\ell\ge 1$, $W$ is fractional in time with covariance $\gamma_0(t)=|t|^{-\alpha_0}$ and in space, it has a covariance $\gamma$ which is a nonnegative function satisfying (\ref{k5}),  the scaling property $\gamma(cx)= c^{-\alpha} \gamma(x)$ for all $c>0$ and $x\in \RR^\ell$ and the structural property $\gamma(\cdot)=\int_{\RR^\ell} K(y-\cdot) K(y)dy$, where $K(\cdot) \ge 0$.  Here $\mathcal{E}$ is the variational constant
	\[
	\mathcal{E}= \sup_{g} \left(\int_0^1\int_0^1 \int_{\RR^{2\ell}}  \frac{\gamma(x-y)}{|s-r|^{\alpha_0}} g^2(s,x) g^2(r,y) dxdydrds - \frac 12
	\int_0^1\int_{\RR^\ell} | \nabla g(r,y)| ^2 dydr \right),
	\]
	where the supremum is taken over all $g$ such that $g(s,\cdot)\in H^1(\RR^\ell)$ and $\int_{\RR^\ell}g^2(s,x)dx=1$ for all $0\le s\le 1$. This result does not cover the case of white-in-time noise. Based on the results in the case of space-time white noise (identity \eqref{id:ceedelta}), Chen conjectures in \cite{Chen} that if the noise is white in time and has spatial covariance $\gamma$ satisfying the previously described conditions, then for every $n=1,2,\dots$,
	\begin{equation}\label{conjecture}
		\lim_{t\to\infty}\frac1t\log\EE u^n(t,x)=n(n^{\frac{2}{2- \alpha}}-1)\tilde\cee
	\end{equation}
	for some constant $\tilde \cee$. We believe that this is an interesting problem and the methods used in Lemma \ref{lem:epeq} are not capable to yield such a precise statement. 
	\end{remark}

	\begin{proof}[Proof of Proposition \ref{prop:phase}]
		(i) We observe that under assumption  \ref{H1} or  \ref{H2}, the result in Theorem \ref{thm:u01} gives
		\begin{equation}\label{tmp:nmmt}
			\cee_n(\lambda\gamma)=\lim_{t\to\infty}\frac1t\log\EE u^n_ \lambda(t,x) 	
		\end{equation}
	for every $\lambda>0$.  We can restrict the supremum in (\ref{id:ehat}) to nonnegative functions which shows that the map $\lambda\mapsto \cee_n(\lambda \gamma)$ is nondecreasing.  It follows that weak phase transition at order $n$ implies (strong) phase transition at the same order. 

	(ii) It  follows directly from Lemma \ref{lem:epeq}, identity \eqref{tmp:nmmt} and the fact that $\lambda\mapsto\cee_n(\lambda \gamma)$ is nondecreasing.

	(iii) It is evident that $\lambda^c_2=\sup\{\lambda>0:\cee_2(\lambda \gamma)=0\}$.  Let $\lambda$ be a positive number such that $\cee_2(\lambda \gamma)=0$. From the expression \eqref{y23} we see that
		\begin{equation}\label{tmp:lh}
			\lambda\le \frac{(2 \pi)^\ell\int_{\RR^{2\ell}}|\xi|^2 |h(\xi)|^2d \xi }{2  \int_{\RR^\ell} (h*h)(-\xi,\xi)\mu(d \xi) } 	
		\end{equation} 
		for all $h\in\mathcal A_{2\ell}$. Fix $s>0$. We choose $h(\xi^1,\xi^2)=\lt(\frac{2 s}{\pi}\rt)^{\frac\ell2} e^{-s (|\xi^1|^2+|\xi^2|^2)}$ for all $\xi^1,\xi^2\in\RR^\ell$, so that $(h*h)(-\xi,\xi)=e^{-s |\xi|^2}$. Hence
		\begin{equation*}
			\int_{\RR^\ell} (h*h)(-\xi,\xi)\mu(d \xi)=\int_{\RR^\ell} e^{-s |\xi|^2}\mu(d \xi)
		\end{equation*}
		and
		\begin{equation*}
			\iint_{\RR^{2\ell}}(|\xi^1|^2+|\xi^2|^2)h^2(\xi^1,\xi^2)d \xi^1d \xi^2=\frac{\ell}{2s}\,.
		\end{equation*}
		It follows from \eqref{tmp:lh} that
		\begin{equation*}
			\lambda\le \frac{\ell (2 \pi)^\ell}{4 \sup_{s>0}s\int_{\RR^\ell} e^{-s |\xi|^2}\mu(d \xi)}\,.
		\end{equation*}
		Passing through the limit $\lambda\uparrow \lambda^c_2$ we obtain the result. 

	Finally, (iv) is a consequence of (iii).
	\end{proof}
	We now prove Theorem \ref{thm:front1}. 
	\begin{proof}[Proof of Theorem \ref{thm:front1}]
		This result is essentially due to \cite{CKK}. However, since the set up in the afore mentioned paper is a little bit different from ours, we sketch the idea here.

		Theorem 1.3 of \cite{CKK} shows that weak phase transition at order 2 happens if and only if condition \eqref{cond:front1} is satisfied. Together with Proposition \ref{prop:phase}, this shows the first statement. It remains to show \eqref{est:lcabove}.
		Following \cite{CKK}, we put $$h_1(r,y)=\frac 1{(2 \pi)^\ell} \int_0^r \int_{\RR^\ell}e^{-\frac s2|\xi|^2}\mu(d \xi)e^{-\frac {4|y|^2}{s}}ds\,. $$
		A further inspection of the proof of Theorem 1.3 in \cite{CKK} reveals that for every $r>0$ and $a>0$, $\cee_2(\lambda \gamma)>0$ for all $\lambda$ such that
		\begin{equation*}
			\lambda\ge \frac{e}{h_1(r,a)}\,.
		\end{equation*} 
		Together with the estimate \eqref{est:e1}, we have
		\begin{equation*}
			\lambda_n^c\le \frac{e}{h_1(r,a)}
		\end{equation*}
		for all $r>0$ and $a>0$. We can pass through the limits $a\to0$ and $r\uparrow\infty$ to obtain \eqref{est:lcabove}.
	\end{proof}
	
	\begin{proof}[Proof of Theorem \ref{thm:front2}]
		As in the proof of Theorem \ref{thm1}, $u_ \lambda(t,x)$ admits the chaos expansion
		\begin{equation*}
			u_ \lambda(t,x)=\sum_{n=0}^\infty \lambda^{\frac n2}I_n(f_n(\cdot,t,x)),
		\end{equation*}
		where $f_n$ is the function defined in \eqref{eq:expression-fn} with $u_0=1$. In addition, we also have
		\begin{equation*}
		n! \|f_n(\cdot,t,x)\|^2_{\HH^{\otimes n}}
		=C^n \int_{[0,t]^n_<} \int_{\RR^n} \prod_{j=1}^n e^{- (s_{j+1} - s_{j}) |\xi^{j} + \cdots + \xi^{1}|^2}\prod_{j=1}^n f(\xi^j) d\xi ds
		\end{equation*}
		for some constant $C$. Using the change of variables $\eta^j=\xi^j+\cdots+\xi^1$, we obtain
		\begin{equation*}
			n! \|f_n(\cdot,t,x)\|^2_{\HH^{\otimes n}}
		=C^n \int_{[0,t]^n_<} \int_{\RR^n} \prod_{j=1}^n e^{- (s_{j+1} - s_{j}) |\eta^j|^2}\prod_{j=1}^n f(\eta^j- \eta^{j-1}) d\xi ds
		\end{equation*}
		with the convention $\eta^0=0$. Using the assumption \ref{H1} we see that
		\begin{equation*}
			\prod_{j=1}^nf(\eta^j- \eta^{j-1})\le \kappa^n \prod_{j=1}^n (f(\eta^j)+f^2(\eta^j))\,.
		\end{equation*}
		Hence, $n! \|f_n(\cdot,t,x)\|^2_{\HH^{\otimes n}}$ is at most
		\begin{align*}
			\left(C \kappa\int_0^\infty\int_\RR e^{-s}|\eta|^2[f(\eta)+f^2(\eta)] d \eta ds\right)^n
			=\left(C \kappa\int_\RR \frac{f(\eta)+f^2(\eta)}{|\eta|^2} d \eta \right)^n\,.
		\end{align*}
		It follows that
		\begin{align*}
			\EE |u_ \lambda(t,x)|^2\le \sum_{n=0}^\infty\left ( \lambda C \kappa \int_\RR \frac{f(\eta)+f^2(\eta)}{|\eta|^2} d \eta\right)^n<\infty,
		\end{align*}
		provided that
		\begin{equation*}
			\lambda< \frac1{C \kappa \int_\RR \frac{f(\eta)+f^2(\eta)}{|\eta|^2} d \eta}\,.
		\end{equation*}
		Hence, $\cee_2(\lambda \gamma)$ vanishes for all $\lambda$ sufficiently small. On the other hand, item (iii) of Proposition \ref{prop:phase} ensures that $\cee_2(\lambda \gamma)>0$ for $\lambda$ sufficiently large. These facts show that phase transition happens. 
	\end{proof}
	We conclude the current section with the following proposition. 
	\begin{proposition} Suppose that phase transition occurs, then for all integers $1<p\le q$, we have
		\begin{equation}
			\frac{p-1}{q-1}\le \frac{\lambda_q^c}{\lambda_p^c} \le 1\,.
		\end{equation}
	\end{proposition}
	\begin{proof}
		It is straightforward from Lemma \ref{lem:epeq}.
	\end{proof}

\section{Appendix}\label{sec:app}
In this section we state and prove several technical lemmas which are used along the paper.
 %\begin{lemma}\label{lem:gebded}
%	Let $\gamma_ \epsilon$ be the function defined in \eqref{eqn:ge}. Then under condition \ref{H2c}, we have
%	\begin{equation*}
%		\sup_{x\in\RR^\ell}|x\cdot \nabla \gamma_ \epsilon(x)|<\infty
%	\end{equation*}
%	for every $\epsilon>0$.
%\end{lemma}
%\begin{proof}
%	Using integration by part, we have
%	\begin{align*}
%		(2 \pi)^\ell x\cdot \nabla \gamma_ \epsilon(x)
%		=\int_{\RR^\ell}\nabla_{\xi}(e^{ix\cdot \xi})\cdot  \xi e^{- \epsilon |\xi|^2}f(\xi)d \xi
%		=(-1)^\ell \int_{\RR^\ell}e^{ix\cdot \xi} \nabla\cdot[\xi e^{- \epsilon |\xi|^2}f(\xi)]d \xi\,.
%	\end{align*}
%	Thus $(2 \pi)^\ell \sup_{x\in\RR^\ell} |x\cdot \nabla \gamma_ \epsilon(x)|$ is at most
%	\begin{align*}
%		\int_{\RR^\ell} \lt|\nabla\cdot[\xi e^{- \epsilon |\xi|^2}f(\xi)]\rt|d \xi
%	\end{align*}
%	which is finite under condition \ref{H2c}. 
%\end{proof}
\begin{lemma}\label{Var BB}
Let $\{B_{0,t}(s), s \in [0,t]\}$ be a Brownian bridge in $\RR^\ell$. Let $0 < s_1 < \dots < s_n < s_{n+1}=t$, then for any $\xi^j\in\RR^\ell, 1 \leq j \leq n$ we have
\begin{eqnarray*}
\EE \left( \sum_{j=1}^n \xi^j \cdot B_{0,t}(s_j) \right)^2 = \sum_{j=1}^n |\xi^1 + \cdots + \xi^j|^2 (s_{j+1}-s_j) - \frac{1}{t} \left|\sum_{j=1}^n (\xi^1 + \cdots  +\xi^j)(s_{j+1}-s_j) \right|^2\,.
\end{eqnarray*}
\end{lemma}
\begin{proof}
Using the covariance of the Brownian bridge, we can write
\begin{eqnarray*}
\EE \left( \sum_{k=1}^n \xi^k \cdot B_{0,t} (s_k)\right)^2 &=&  
 \sum_{j,k=1}^n \xi^{k}\cdot \xi^j (s_k\wedge s_j -\frac{1}{t} s_k s_j)\\
&=&\frac{1}{t}\left[ \sum_{j=1}^n |\xi^j|^2 s_j(t-s_j)+ 2 \sum_{1\leq j < k \leq n} \xi^k\cdot \xi^j s_j (t-s_k) \right]\,.
\end{eqnarray*}
Let $s_0=0$ and write $\Delta s_j= s_{j+1}-s_j$. Then,  the above expression can be written as
\begin{eqnarray*}
\EE \left( \sum_{k=1}^n \xi^k  \cdot B_{0,t}(s_k)\right)^2 &=& \frac{1}{t} \Bigg[\sum_{j=1}^n |\xi^j|^2 \sum_{0 \leq p \leq j-1} \Delta s_p \sum_{j \leq q \leq n} \Delta s_q \\
&& \qquad  + 2 \sum_{0 \leq j < k \leq n} \xi^k\cdot \xi^j \sum_{0 \leq p \leq j-1} \Delta s_p \sum_{k \leq q \leq n} \Delta s_q \Bigg]\\
&=&\frac{1}{t}\left[\sum_{0 \leq p < q \leq n} \left ( \sum_{j=p+1} ^q |\xi^j|^2 + \sum_{p+1 \leq j < k \leq q} 2\xi^k\cdot \xi^j \right) \Delta s_p \Delta s_q \right]\,.
\end{eqnarray*}
Write $\eta^q = \xi^1+ \cdots + \xi^q$,  $\eta^0=0$. Then we obtain 
\begin{eqnarray*}
\EE \left( \sum_{k=1}^n \xi^k \cdot  B_{0,t}(s_k)\right)^2& =& \frac{1}{t}\left[\sum_{0 \leq p < q \leq n} \left |\eta^q - \eta^p\right|^2 \Delta s_p \Delta s_q \right]\\
&=&\frac{1}{2t} \left[\sum_{p,q = 0}^n |\eta^p|^2 \Delta s_p \Delta s_q +\sum_{p,q = 0}^n |\eta^q|^2 \Delta s_p \Delta s_q  - \sum_{p,q =0}^n  \eta^p\cdot \eta^q \Delta s_p \Delta s_q \right]\\
&=& \sum_{j=1}^n |\xi^1 + \cdots + \xi^j|^2 (s_{j+1}-s_j) - \frac{1}{t} \left|\sum_{j=1}^n (\xi^1 + \cdots +\xi^j)(s_{j+1}-s_j) \right|^2\,,
\end{eqnarray*}
which proves the lemma. 
\end{proof}
 \begin{lemma}\label{lem:Sub add BB}
Assume that $\gamma$ satisfies conditions \ref{H1} or \ref{H2}. For each $r>0$, let $B_{0,r}$ be a Brownian bridge in $\RR^\ell$. Then we have for any $t_1,t_2 > 0$, we have
\begin{eqnarray*}
\EE \exp \left(\int_0^{t_1+t_2} \gamma(B_{0,t_1+t_2}(s)) ds \right) \leq \EE \exp \left(\int_0^{t_1} \gamma(B_{0,t_1}(s)) ds \right) \EE \exp \left(\int_0^{t_2} \gamma(B_{0,t_2}(s)) ds \right)\,.
\end{eqnarray*}
In other words, the map $t\mapsto T(t):={ \log}\EE \exp \left(\int_0^{t} \gamma(B_{0,t}(s)) ds \right)$ is subadditive. 
\end{lemma}
 \begin{proof}
Thanks to Proposition \ref{prop:expbridge}, we can assume that $\gamma:\RR^\ell\to\RR$ is a bounded positive definite function.  Let $\mathcal{G}_{r}$ be the filtration generated by $B_{0,t_1+t_2}$ from $0$ to $r$, then we have
\begin{multline}\label{tmp:bt1t2}
\EE \exp \left(\int_0^{t_1+t_2} \gamma(B_{0,t_1+t_2}(s)) ds \right)\\
=\EE \left [ \exp \left( \int_0^{t_1} \gamma(B_{0,t_1+t_2}(s)) ds \right) \EE\left( \exp \left( \int_{t_1}^{t_1+t_2} \gamma(B_{0,t_1+t_2}(s)) ds \right)\Big| \mathcal{G}_{t_1}\right)\right]\,.
\end{multline}
Applying Lemmas \ref{lem:BB decomp} and \ref{lem:gamma}, we see that
\begin{align*}
	\EE\left( \exp \left( \int_{t_1}^{t_1+t_2} \gamma(B_{0,t_1+t_2}(s)) ds \right)\Big| \mathcal{G}_{t_1}\right)
	&\le\sup_{x\in\RR^\ell}\EE \exp \left( \int_{t_1}^{t_1+t_2} \gamma(B_{t_1,t_1+t_2}(s)+x) ds \right)
	\\&=T(t_2)\,.
\end{align*}
Observe that $\{B_{0,t_1+t_2}(t_1+t_2-s),s\in[0,t_1+t_2]\}$ is also a Brownian bridge, thus together with a change of variable, we have
\begin{align*}
	\EE  \exp \left( \int_0^{t_1} \gamma(B_{0,t_1+t_2}(s))  ds\right)=\EE  \exp \left( \int_{t_2}^{t_1+t_2} \gamma(B_{0,t_1+t_2}(s))  ds\right)\,.
\end{align*}
By the same argument as above, the above expectation is at most $T(t_1)$. Together with \eqref{tmp:bt1t2}, we obtain the result.
 \end{proof}

\begin{lemma}\label{Lemma:Conv chaos} 
We assume either condition \ref{H1} or \ref{H2}.  For each $t>0$ and $n\ge1$, we define
\begin{eqnarray*}
I_n(t) &=& D^n\int_{[0,t]_<^n} \int_{\RR^{n\ell}} \exp \left\{\frac{1}{t} \Big|\sum_{i=1}^n (s_{i+1} -s_{i}) (\xi^{i}+\cdots+\xi^{1} ) \Big|^2 \right.\\
&&\left.- \sum_{i=1}^n(s_{i+1} - s_{i}) |\xi^{i} + \cdots + \xi^{1}|^2 \right\}    \mu(d\xi) ds\,,
\end{eqnarray*}
for some constant $D\in \RR$, where $\mu(d\xi) =\prod_{j=1}^n \mu(d\xi^j)$ and $ds=ds_1\cdots ds_n$. Then
\begin{equation}
\sum_{n=0}^{\infty} I_n(t)\leq c_1 \exp (c_2 t)
\end{equation}
%When $\mu $ satisfies conditions \ref{H1}, then, the series $\sum_{n=0}^{\infty} I_n(t) $ converges and
%\begin{equation}  \label{y7}
%\sum_{n=1}^\infty I_n(t) \le     c_1 |D|(1+ t^{\frac 52})\exp\left(c_2 t|D|\right)\,,
%\end{equation}
where  $c_1$ and $c_2$ are two constants depending on $D$.
% When $\mu $ satisfies condition \ref{H2}, the series $\sum_{n=0}^{\infty} I_n(t)$ also converges and $\sup_{0<t\le 1}\sum_{n=0}^{\infty} I_n(t)$ is finite.
\end{lemma}          

\begin{proof}
	Fix $\epsilon>0$, let $I_n^ \epsilon(t)$ be defined as $I_n(t)$ but with $\mu(d \xi)$ being replaced by $e^{-\epsilon|\xi|^2}\mu(d \xi)$. 

	\textit{Step 1:} We claim that
	\begin{equation*}
		\sum_{n=0}^\infty I_n^\epsilon(t)=\EE\exp\lt\{D\int_0^t \gamma_ \epsilon(\sqrt2 B_{0,t})ds  \rt\}\,.
	\end{equation*}
	Indeed, for each $n\ge0$, applying Lemma \ref{Var BB}, we have
	\begin{align} \notag
	\EE\left[   \int_0^t \gamma_ \epsilon(\sqrt 2 B_{0,t})ds \right]^n
	% &= \EE \left [ \int_0^t \int_{\RR^\ell} e^{i \xi \cdot(B_{0,t}^1(s)- B_{0,t}^2(s))} \mu_{\epsilon}(d\xi) ds\right]^n\\ \notag
	&=\int_{[0,t]^n} \int_{(\RR^\ell)^n} \EE \exp\left\{i\sqrt2 \sum_{k=1}^n \xi^k \cdot B_{0,t}(s_k)  \right\}  \mu_\epsilon(d\xi) ds\\ \notag
	% &=\int_{[0,t]^n} \int_{(\RR^\ell)^n} \left| \EE \exp\left\{i \sum_{k=1}^n \xi^k \cdot B^1_{0,t}(s_k)  \right\} \right|^2 \mu_\epsilon(d\xi) ds\\ \notag
	&=n!\int_{[0,t]^n_<} \int_{(\RR^\ell)^n} \exp\left\{-{\rm Var} \left(  \sum_{k=1}^n \xi^k \cdot B_{0,t}(s_k)  \right) \right\}\mu_\epsilon(d\xi) ds\\ \notag
	&=n!\int_{[0,t]^n_<} \int_{(\RR^\ell)^n} \exp\Bigg\{ \sum_{j=1}^n - |\xi^1 + \cdots + \xi^j|^2 (s_{j+1}-s_j) \\
	& +\frac{1}{t} \left|\sum_{j=1}^n (\xi^1 + \cdots+ \xi^j)(s_{j+1}-s_j) \right|^2 \Bigg\}\mu_\epsilon(d\xi) ds\,,
	\label{eqn:d.int.form}
	\end{align}
	where $[0,t]^n_<$ is defined in \eqref{k8} and we use the notations
 	$\mu_\epsilon(d\xi)  = \prod_{k=1}^n e^{-\epsilon |\xi^k|^2}\mu(d\xi^k)$ and  $ds = ds_1 \cdots ds_n$. Together with the  Taylor expansion of $e^x$, we obtain the claim. 

 	\textit{Step 2:} We show that for every $t>0$
 	\begin{equation*}
 		\sup_{\epsilon>0}\EE\exp\lt\{2D\int_0^t \gamma_ \epsilon(\sqrt 2 B(s))ds \rt\}\le c_1 e^{c_2t} 
 	\end{equation*}
 	for some positive constant $c_1,c_2$.
 	In fact, using Taylor's expansion,
 \begin{align*}
  \EE \exp \left(2 D \int_0^{t} \gamma_{\epsilon}(\sqrt{2} B(s)) ds \right) 
  \le \sum_{k=0}^{\infty} \frac{(2D)^k}{k!} \int_{[0,t]^k} \int_{\RR^{\ell k}} \prod_{j=1}^k e^{-(s_{\sigma(i+1)}-s_{\sigma(i)})|\xi^{\sigma(j)}+ \cdots + \xi^{\sigma(1)}|^2} \mu(d\xi) ds,
 \end{align*}
where $\sigma$ denotes the permutation of $\{1,2, \dots, k\}$ such that $0 < s_{\sigma(1)} < \cdots < s_{\sigma(k)} < t$, and we denote $s_{\sigma(k+1)} = t$. We treat conditions \ref{H1} and \ref{H2} separately. 

{\it Case 1:} under the condition  \ref{H2}. 
Using the assumption that $\gamma$ is positive and positive definite {and an argument similar to \cite{HHNT}*{estimate (3.7)} we conclude }that the last expression above is bounded by

 %{\color {blue}  Maybe here we could explain how the fact that $\gamma$ is a positive function o measure  implies that we can replace $\xi^j+ \cdots+\xi ^1$ by just $\xi^j$, or give a reference (for instance our paper in EJP) where this property is proved}
 \begin{eqnarray*}
&& \sum_{k=0}^{\infty} \frac{(2D)^k}{k!} \int_{[0,t]^k} \int_{\RR^{\ell k}} \prod_{j=1}^k e^{-(s_{\sigma(i+1)}-s_{\sigma(i)})|\xi^{\sigma(j)}|^2} \mu(d\xi) ds\\
 &=& \sum_{k=0}^{\infty} (2D)^k\int_{[0,t]_<^k} \int_{\RR^{\ell k}} \prod_{j=1}^k e^{-(s_{i+1}-s_{i})|\xi^j|^2} \mu(d\xi) ds\\
 &\leq& e^{ct} \sum_{k=0}^{\infty} (2D)^k\int_{[0,t]_<^k} \int_{\RR^{\ell k}} \prod_{j=1}^k e^{-(s_{i+1}-s_{i})(c+|\xi^j|^2)} \mu(d\xi) ds \\
 &\leq& e^{ct} \sum_{k=0}^{\infty} (2D)^k \int_{\RR^{\ell k}} \prod_{j=1}^k \frac{1}{c+|\xi^j|^2} \mu(d\xi)\,,
 \end{eqnarray*}
 the series is convergent when $c$ is sufficiently {large, which  completes }the proof. 

{\it Case 2:}  Under the condition \ref{H1}.  Then from the above we see that
\begin{align*}
\EE\exp\lt\{2D\int_0^t \gamma_ \epsilon(\sqrt 2B(s))ds \rt\}
 % &\leq&  \sum_{k=0}^{\infty} \frac{(2D)^k}{k!} \int_{[0,t]^k} \int_{\RR^{k}} \prod_{j=1}^k e^{-(s^{\sigma(i+1)}-s^{\sigma(i)})|\xi^{\sigma(j)}+ \cdots + \xi^{\sigma(1)}|^2} \mu(d\xi) ds\\
\leq\sum_{k=0}^{\infty} (2D)^k \int_{[0,t]_<^k} \int_{\RR^{ k}} \prod_{j=1}^k e^{-(s_{i+1}-s_i)|\xi^{j}+ \cdots + \xi^{1}|^2} \mu(d\xi) ds\,.
\end{align*}
 Making the change of variables $\xi^{j}+ \cdots + \xi^{1} \to \eta^j$ for $j = 1, \dots, k$ together with condition \ref{H1}, the above series is at most 
 \begin{eqnarray*}
  && \sum_{k=0}^{\infty} (2\kappa D)^k \int_{[0,t]_<^k} \int_{\RR^{k}} \prod_{j=1}^k e^{-(s_{i+1}-s_i)|\eta^j|^2} \prod_{j=1}^k f(\eta^j-\eta^{j-1}) d\eta ds\\
 &\leq& \sum_{k=0}^{\infty} (2 D)^k \int_{[0,t]_<^k} \int_{\RR^{k}} \prod_{j=1}^k e^{-(s_{i+1}-s_i)|\eta^j|^2} \prod_{j=1}^k f(\eta^j-\eta^{j-1}) d\eta ds\\
 &\leq&\sum_{k=0}^{\infty} (2 \kappa D)^k \int_{[0,t]_<^k} \int_{\RR^{k}} \prod_{j=1}^k e^{-(s_{i+1}-s_i)|\eta^j|^2} f(\eta^1)\sum_{\alpha \in \mathcal{D}_k} \prod_{j=1}^k f(\eta^j)^{\alpha_j} d\eta ds\\
 &\leq&e^{ct}\sum_{k=0}^{\infty} (2 \kappa D)^k \int_{[0,t]_<^k} \int_{\RR^{k}} \prod_{j=1}^k e^{-(s_{i+1}-s_i)(c+ |\eta^j|^2)} f(\eta^1)\sum_{\alpha \in \mathcal{D}_k} \prod_{j=1}^k f(\eta^j)^{\alpha_j} d\eta ds\\
 &\leq&e^{ct}\sum_{k=0}^{\infty} (2 \kappa D)^k \int_{\RR^{k}} \prod_{j=1}^k \frac{1}{c+ |\eta^j|^2} f(\eta^1)\sum_{\alpha \in \mathcal{D}_k} \prod_{j=1}^k f(\eta^j)^{\alpha_j} d\eta ds
 \end{eqnarray*}
 where $c$ is a positive constant to be chosen later, $\mathcal{D}_k$ is a subset of multi-indices of length $k-1$ and $\text{Card}(\mathcal{D}_k)=2^{k-1}$ and for any $\alpha \in \mathcal{D}_k$, $\alpha_i \in \{0, 1, 2\}$, $i = 2,\dots, k$ but $\alpha_1 \in \{0,1\}$.  Using the inequality 
 \begin{equation*}
 \int_{\RR}\frac{f(\eta)}{c+\eta^2} d\eta \leq \left( \int_{\RR} \frac{1}{c+\eta^2} d\eta \right)^{1/2} \left( \int_{\RR} \frac{f^2(\eta)}{c+\eta^2} d\eta \right)^{1/2}
 \end{equation*}
 if necessary, we conclude that 
 \begin{eqnarray*}
  \sum_{n=0}^{\infty} I_n &\leq& e^{ct} \sum_{k=0}^{\infty} (4\kappa D)^k \left(\int_{\RR}\frac{f^2(\eta)}{c+\eta^2}d\eta \right)^k \,,
 \end{eqnarray*}
we see that when $c$ is sufficiently big, the above series is convergent.

\textit{Step 3:}
By Cauchy-Schwarz inequality and the fact that $B_{0,t}\stackrel{\textrm{law}}{=}B_{0,t}(t-\cdot)$ we have
 	\begin{align*}
 		\EE\exp\lt\{D\int_0^t \gamma_ \epsilon(\sqrt 2 B_{0,t}(s))ds \rt\}\le \EE\exp\lt\{2D\int_0^{\frac t2}\gamma_ \epsilon(\sqrt 2B_{0,t}(s))ds \rt\}\,.
 	\end{align*}
 	Together with \eqref{id:density}, we arrive at
 	\begin{equation}\label{est:brbyb}
 		\EE\exp\lt\{D\int_0^t \gamma_ \epsilon(\sqrt 2 B_{0,t}(s))ds \rt\}\le\EE\exp\lt\{2D\int_0^{\frac t2}\gamma_ \epsilon(\sqrt 2B(s))ds \rt\}\,.
 	\end{equation}
 	It follows from previous steps and \eqref{est:brbyb} that for all $t>0$
 	\begin{equation*}
 		\sup_{\epsilon>0}\sum_{n=0}^\infty I_n^ \epsilon(t)=\sup_{\epsilon>0}\EE\exp\lt\{D\int_0^t \gamma_ \epsilon(\sqrt 2 B_{0,t}(s))ds \rt\} \le c_1e^{c_2t}\,.
 	\end{equation*}
 	On the other hand, by monotone convergence theorem, it is evident that
 	\begin{equation*}
 		\sup_{\epsilon>0}\sum_{n=0}^\infty I_n^ \epsilon(t)=\sum_{n=0}^\infty I_n(t)\,.
 	\end{equation*}
This completes the proof. 
\end{proof}

\begin{lemma}\label{lem:ldev}
Let $A_M^+=\{(y^1,\dots,y^n)\in([0,\infty)^\ell)^n:|y^j-y^k|\le M \text{ for all } 1\le j<k\le n \}$ and $\bar \alpha=(\alpha ,0,\dots,0)\in\RR^\ell$. Then
	\begin{align}   \label{k1}
		\lim_{t\to\infty}\frac1t\log\int_{A_M^+}\exp\left\{-\sum_{j=1}^n\left(\frac{|y^j-  \bar\alpha t|^2}{2t}+\beta |y^j|\right)\right\}dy
		= \left\{
		\begin{array}{lll}
			\frac {n\beta^2}2- n\beta \alpha & \text{ if } & \alpha> \beta\\
			 -\frac {n\alpha^2}{2} & \text{ if } & \alpha\le \beta\,.
		\end{array}
		\right.
	\end{align}
\end{lemma}
\begin{proof}
	We first show the upper bound of \eqref{k1}. We observe that
	\begin{equation*}
		\int_{A_M^+}\exp\left\{-\sum_{j=1}^n\left(\frac{|y^j- \bar\alpha t|^2}{2t}+\beta |y^j|\right)\right\}dy\le\int_{(\RR^\ell)^n}\exp\left\{-\sum_{j=1}^n\left(\frac{|y^j-  \bar\alpha t|^2}{2t}+\beta |y^j_1|\right)\right\}dy\,.
	\end{equation*}
	By a change of variables, the integral on the right hand side of the above inequality is the same as
	\begin{equation*}
		t^{n\ell}\int_{(\RR^\ell)^n}\exp\left\{-t\sum_{j=1}^n\left(\frac{|y^j_1- \alpha |^2}{2}+\beta |y^j_1|+ \sum_{k=2}^\ell \frac{|y^j_k|^2}{2}\right)\right\}dy\,.
	\end{equation*}
	Applying Laplace principle, we see that the limit
	\begin{equation*}
		\lim_{t\to\infty}\frac1t\log\int_{(\RR^\ell)^n}\exp\left\{-\sum_{j=1}^n\left(\frac{|y^j-  \bar\alpha t|^2}{2t}+\beta |y^j_1|\right)\right\}dy
	\end{equation*}
	equals to the right hand side of \eqref{k1}. This shows the upper bound in \eqref{k1}.

	To show the lower bound, we consider two cases.

	\textbf{Case 1:} $\ell=1$. We first notice that the integral in  (\ref{k1}) can be written as
	\begin{align*}
		e^{(\frac{n\beta^2}2-n\alpha \beta)t }I(t)
	\end{align*}
	where
	\begin{align*}
		I(t)= \int_{A_M^+}\exp\left\{-\sum_{j=1}^n\frac{(y^j-(\alpha- \beta)t)^2}{2t}\right\}dy\,.
	\end{align*}
	We consider the case $\alpha>\beta$. For each $j=2,\dots,n$, we have $(y^j-(\alpha- \beta)t)^2\le2(y^j-y^1)^2+2(y^1-(\alpha- \beta)t)^2$. We then write
	\begin{align*}
		I(t)&\ge \int_{0}^\infty  e^{-\frac{n-1}{t}(y^1-(\alpha- \beta)t)^2} \prod_{j=2}^n\left( \int_{y^j\ge0, |y^j-y^1|\le M/2}e^{-\frac{1}t (y^j-y^1)^2} dy^j\right)  dy^1
		\\&\ge e^{-\frac{M^2}{4t}} \int_{0}^\infty  e^{-\frac{n-1}{t}(y^1-(\alpha- \beta)t)^2} h(y^1) dy^1,
	\end{align*}
	where
	\begin{align*}
		h(y^1)=\left( \int_{y\ge0, |y-y^1|\le M/2} dy\right)^{n-1}\,.
	\end{align*}
	We continue to estimate
	\begin{align*}
		I(t)\ge e^{-\frac{M^2}{4t}}\int_{(\alpha- \beta)t}^{2(\alpha- \beta)t} e^{-\frac{n-1}{t}(y^1-(\alpha- \beta)t)^2} h(y^1) dy^1
	\end{align*}
	On the region $[(\alpha- \beta)t,2(\alpha- \beta)t]$,  assuming that $t$ satisfies $\alpha -\beta \ge \frac M{2t}$, we have $h(y^1)=M^{n-1}$, thus
	\begin{align*}
		I(t)&\ge M^{n-1} e^{-\frac{M^2}{4t}}\int_{(\alpha- \beta)t}^{2(\alpha- \beta)t} e^{-\frac{n-1}{t}(y^1-(\alpha- \beta)t)^2}  dy^1
		\\&\ge tM^{n-1} e^{-\frac{M^2}{4t}}\int_{(\alpha- \beta)}^{2(\alpha- \beta)} e^{-(n-1)t(z-(\alpha- \beta))^2}  dz.
	\end{align*}
	It follows from Laplace principle that
	\begin{equation*}
		\liminf_{t\to\infty}\frac1t\log I(t)\ge 0\,.
	\end{equation*}

	We now consider the case $\alpha\le \beta$. We write
	\begin{align*}
		\int_{A_M^+}\exp\left\{\sum_{j=1}^n\left(-\frac{(y^j- \alpha t)^2}{2t}-\beta y^j\right)\right\}dy
		=e^{-\frac{n\alpha^2}2 t} J(t),
		\end{align*}
		where
		\[
		J(t)=	
		\int_{A_M^+}\exp\left\{-\frac1{2t}\sum_{j=1}^n(|y^j|^2+2(\beta- \alpha)ty^j)\right\}dy.
 \]
	We can write
	\begin{align*}
	J(t) \ge 	\int_{0}^te^{-\frac1{2t}((2(n-1)) |y^1|^2+2n (\beta- \alpha)ty^1)} g(y^1) dy^1,
	\end{align*}
	where
	\begin{equation*}
		g(y^1)=\left(\int_{y\ge0,|y-y^1|\le M/2} e^{-\frac{1}{2t}(2|y-y^1|^2+2(\beta- \alpha)t(y-y^1))}dy\right)^{n-1}
	\end{equation*}
	On $[0,t]$, $g(y^1)$ is at least $(\frac M2 e^{-\frac1{2t}  (\frac {M^2} 2 + (\beta -\alpha) Mt)}  )^{n-1} $. Thus for every $t\ge1$,
	\[
	J(t) \ge  (\frac M2 e^{-\frac1{2t}  (\frac {M^2} 2 + (\beta -\alpha) Mt)}  )^{n-1} 
	t\int_{0}^ 1e^{- t   \left(  (n-1) y^2  + n (\beta- \alpha)y \right)} dy
	\]
	As before, by the Laplace principle,
	\[
	\liminf_{t\to\infty}\frac1t\log J(t)\ge 0\,.
	\]

	\textbf{Case 2:} $\ell\ge2$. For each $1\le j\le n$, applying the inequality $|y^j|\le \sum_{k=1}^\ell |y^j_k|$ we see that $\exp\left\{-\sum_{j=1}^n\left(\frac{|y^j- \bar\alpha t|^2}{2t}+\beta |y^j|\right)\right\}$ is at least
	\begin{align*}
		\exp\lt\{-\sum_{j=1}^n \lt(\frac{(y^j_1- \alpha t)^2}{2t}+\beta|y^j_1|\rt) \rt\}\prod_{k=2}^\ell \exp\left\{-\sum_{j=1}^n\left(\frac{|y^j_k|^2}{2t}+\beta |y^j_k|\right)\right\}\,.
	\end{align*}
	In addition, the domain $A_M^+$ contains the region $$\{(y^1,\dots,y^n)\in([0,\infty)^\ell)^n:|y^j_k-y^p_k|\le M' \text{ for all } 1\le k\le \ell \text{ and } 1\le j<p\le n\} $$ for some positive constant $M'$. The above region can be written as $A^\ell$ where $A$ is the set $\{z=(z^1,\dots,z^n) \in[0,\infty)^n:|z^j-z^p|\le M'\mbox{ for all }1\le j<p\le n\}$. Therefore, the left hand side of \eqref{k1} is at least
	\begin{multline*}
		\liminf_{t\to\infty}\frac1t\log\int_{ A}\exp\lt\{-\sum_{j=1}^n \lt(\frac{(y^j_1- \alpha t)^2}{2t}+\beta|y^j_1|\rt) \rt\}dy^1\cdots dy^n
		\\
		+\sum_{k=2}^\ell\liminf_{t\to\infty}\frac1t\log\int_{ A}\exp\left\{-\sum_{j=1}^n\left(\frac{|y^j_k|^2}{2t}+\beta |y^j_k|\right)\right\}dy^1\cdots dy^n\,.
	\end{multline*}
	We can apply the result in Case 1 to conclude the proof. 
\end{proof}

\noindent
	\textbf{Acknowledgment:} The authors thank L. Chen for sending them  the manuscript \cite{CKK} which   inspired Section \ref{sec:phase}.    	J. Huang and K. L\^e were  supported by the NSF  Grant no. 0932078 000, while  they visited the Mathematical Sciences Research Institute in Berkeley, California, during the Fall 2015 semester.  D. Nualart was supported by the NSF grant no.  DMS1512891 and the ARO grant no. FED0070445. {We would like to thank two anonymous referees for careful reading of the paper and many constructive suggestions. }

\end{document}